\documentclass[12pt,english]{article}
\usepackage[T1]{fontenc}
\usepackage[latin9]{inputenc}
\usepackage{geometry}
\usepackage{amsmath}
\usepackage{amsthm}
\usepackage{amssymb}
\usepackage{setspace}
\makeatletter

\providecommand{\tabularnewline}{\\}

\numberwithin{equation}{section}
\numberwithin{figure}{section}
\theoremstyle{plain}
\newtheorem{thm}{\protect\theoremname}[section]
  \theoremstyle{remark}
  \newtheorem{rem}[thm]{\protect\remarkname}
  \theoremstyle{definition}
  \newtheorem{example}[thm]{\protect\examplename}
  \theoremstyle{plain}
  \newtheorem{fact}[thm]{\protect\factname}
  \theoremstyle{plain}
  \newtheorem{lem}[thm]{\protect\lemmaname}
  \theoremstyle{plain}
  \newtheorem{prop}[thm]{\protect\propositionname}

\date{}
\usepackage[titletoc, title]{appendix} 
\usepackage{hyperref}
\hypersetup{
                    colorlinks=true,
                    linkcolor=red,
                    citecolor=blue,
                    urlcolor=green,
                    linktocpage=true,
                    pdfstartview=FitH,
                }

\makeatother

\usepackage{babel}
  \providecommand{\examplename}{Example}
  \providecommand{\factname}{Fact}
  \providecommand{\lemmaname}{Lemma}
  \providecommand{\propositionname}{Proposition}
  \providecommand{\remarkname}{Remark}
\providecommand{\theoremname}{Theorem}

\begin{document}
\global\long\def\C{\mathbb{C}}

\global\long\def\R{\mathbb{R}}

\global\long\def\Z{\mathbb{Z}}

\global\long\def\T{\mathbb{T}}

\global\long\def\Q{\mathbb{Q}}

\global\long\def\F{\mathbb{F}}

\global\long\def\N{\mathbb{N}}

\global\long\def\Sph{\mathbb{S}}

\global\long\def\RP{\mathbb{RP}}

\global\long\def\GL{{\rm GL}}

\global\long\def\GPL{{\rm PGL_{n+1}}\left(\R\right)}

\global\long\def\CP{\mathbb{CP}}

\global\long\def\lindep{\text{linearly independent}}

\global\long\def\K{{\cal K}}

\global\long\def\A{{\cal A}}

\global\long\def\W{{\cal W}}

\global\long\def\L{{\cal L}}

\global\long\def\ovec#1{\overrightarrow{#1}}

\global\long\def\dvec{\ovec d}

\global\long\def\cvec{\ovec c}

\global\long\def\sp{{\rm sp}}

\global\long\def\Psp{{\rm \overline{sp}}}

\global\long\def\iprod#1#2{\langle#1,\,#2\rangle}

\global\long\def\phee{\varphi}

\global\long\def\Px{\varphi\left(\bar{x}\right)}

\global\long\def\Py{\varphi\left(\bar{y}\right)}

\global\long\def\Po{\varphi\left(\bar{0}\right)}

\global\long\def\ta{\widetilde{a}}

\global\long\def\tb{\widetilde{b}}

\global\long\def\tF{\widetilde{F}}

\global\long\def\sub{\subseteq}

\title{The Fundamental Theorems of Affine and Projective Geometry Revisited}

\author{Shiri Artstein-Avidan\thanks{Supported by ISF grant No. 665/15}
~and Boaz A. Slomka\thanks{Corresponding author}}

\maketitle
\noindent \begin{center}
\begin{tabular}{ccccccc}
School of Mathematical Sciences  &  &  &  &  &  & Department of Mathematics \tabularnewline
Tel Aviv University &  &  &  &  &  & University of Michigan\tabularnewline
Tel Aviv 69978 Israel &  &  &  &  &  & Ann Arbor, MI 48109-1043 U.S.A\tabularnewline
Email: shiri@post.tau.ac.il &  &  &  &  &  & Email: bslomka@umich.edu\tabularnewline
 &  &  &  &  &  & \tabularnewline
\end{tabular}
\par\end{center}
\begin{abstract}
The fundamental theorem of affine geometry is a classical and useful
result. For finite-dimensional real vector spaces, the theorem roughly
states that a bijective self-mapping which maps lines to lines is
affine. In this note we prove several generalizations of this result
and of its classical projective counterpart. We show that under a
significant geometric relaxation of the hypotheses, namely that only
lines parallel to one of a fixed set of finitely many directions are
mapped to lines, an injective mapping of the space must be of a very
restricted polynomial form. We also prove that under mild additional
conditions the mapping is forced to be affine-additive or affine-linear.
For example, we show that five directions in three dimensional real
space suffice to conclude affine-additivity. In the projective setting,
we show that $n+2$ fixed projective points in real $n$-dimensional
projective space , through which all projective lines that pass are
mapped to projective lines, suffice to conclude projective-linearity.
\end{abstract}


\newpage{}\tableofcontents{}

\section{Introduction}

\subsection{Overview}

Additive, linear, and affine maps play a prominent role in mathematics.
One of the basic theorems concerning affine maps is the so-called
``fundamental theorem of affine geometry'' which roughly states
that if a bijective map $F:\R^{n}\to\R^{n}$ maps any line to a line,
then it must be an affine transformation, namely of the form $x\mapsto Ax+b$
where $b\in\R^{n}$ is some fixed vector and $A\in GL_{n}(\R)$ is
an invertible linear map. Its projective counterpart, which is called
the ``fundamental theorem of projective geometry'', states that
a map $F:\RP^{n}\to\RP^{n}$ which maps any projective line to a projective
line, must be a projective linear transformation.

These statements have been generalized and strengthened in numerous
ways, and we present various precise formulations, together with references
and other historical remarks, in Section \ref{sec:history}.

While most generalizations regard relaxing the bijectivity conditions,
replacing assumptions on lines with collinearity preservation, or
showing that the proofs can be adjusted so that they work over fields
other than $\R$, in this paper we will be interested in a geometric
relaxation. Instead of assuming that \textit{all} lines are mapped
to lines, one can consider some sub-family of lines, and demand only
that lines in this sub-family are mapped onto (or into) lines.

We show that indeed, in the projective setting, it suffices to assume
the condition line-to-line for a subfamily of lines consisting of
all lines passing through some fixed $n+2$ generic projective points.
This is formulated in Theorem \ref{thm:Pr+2} below. Another interesting
case is when the points are not generic, and $n+1$ of them lie on
a hyperplane. This corresponds in a sense to a case which comes up
in the affine setting, in which all parallel lines in $n+1$ fixed
generic directions are mapped to parallel lines. This (affine) result
was exhibited in \cite{ArtsteinSlomka2011}.

The general situation in affine geometry is somewhat different, since
in a sense (to be explained below) parallelism is lost, and one may
find examples of families of lines in $n+1$ directions in $\R^{n}$
all mapped to lines in a non-linear manner, see Example \ref{exm:R3}.
However, we show that already with $n$ generic directions in which
lines are being mapped to lines, the mapping must be of a very restricted
polynomial form. This result is given in Theorem \ref{thm:poly-form}.
In Theorem \ref{thm:Poly_n+1} we analyze the further restrictions
on such maps, arising from an additional $\left(n+1\right)^{{\rm th}}$
direction in which lines are mapped to lines. As a consequence, we
describe necessary conditions together with several examples in which
such maps are forced to be affine-additive. In particular, for $n=3$
we prove that five directions, each three of which are linearly independent,
suffice to conclude affine-additivity. This is given as Theorem \ref{thm:NFTAG3D}.
Additional interesting cases for general dimensions are given in Section
\ref{sub:affine-additive}.

We remark that there is no continuity assumption in any of the main
results of this paper. However, adding such an assumption allows one
to deduce affine-linearity instead of affine-additivity. In Section
\ref{sec:continuousFTAGproof} we discuss this, and other interesting
consequences of a continuity assumption.

Our results may be applied, for example, in order to simplify the
proofs of several known results, such as Alexandrov's characterization
of Lorentz transformations \cite{Ale67}, as well as Pfeffer's generalization
\cite{Pfef81} for higher dimensions. These, and further applications
will be discussed elsewhere.

\subsection{Notations\label{sec:Notations}}

To formally state our results, it will be useful to introduce some
notation. Throughout this note, $\{e_{1},e_{2},...,e_{n}\}$ will
denote the standard orthonormal basis of $\R^{n}$, and $v=e_{1}+\dots+e_{n}$
will denote their vector sum. By a line in $\R^{n}$ we will mean
a translation of a one-dimensional subspace, so it can be written
as $a+\R b$ where $a,b\in\R^{n}$ are fixed vectors, in which case
we shall say that the line is parallel to $b$ or in direction $b$.
We denote the family of all lines parallel to some vector of a given
set $v_{1},...,v_{k}\in\R^{n}$ by $\L(v_{1},...,v_{k})$. Finally,
we call $k$ vectors $j$-independent if each $j$ of them are linearly
independent, and when a set of vectors is $n$-independent we sometimes
say that they are in general position, or generic.

Denote the projective real $n$-space by $\RP^{n}$. We shall signify
a projective point $\bar{p}\in\RP^{n}$ by a bar mark. Often, the
point $\bar{p}$ will correspond to one of its lifts $p\in\R^{n+1}$
according to the standard projection $\R^{n+1}\setminus\left\{ 0\right\} \to\RP^{n}$.
Given projective points $\bar{p_{1}},\bar{p_{2}},...,\bar{p_{k}}\in\RP^{n}$,
denote their projective span by $\Psp\{\bar{p_{1}},\bar{p_{2}},...,\bar{p_{k}}\}$,
that is, the projective subspace of least dimension, containing these
point. In particular, we denote the $n+1$ projective points in $\R P^{n}$
corresponding to the lines passing through the standard basis $e_{1},...,e_{n+1}$
of $\R^{n+1}$ by $\bar{e}_{1},...,\bar{e}_{n+1}$. We say that projective
points $\bar{a}_{1},\bar{a}_{2},...,\bar{a}_{m}$ in $\RP^{n}$ are
in general position (or generic) if each $k\leq n+1$ of (any of)
their corresponding lifts $a_{1},a_{2},...,a_{m}\in\R^{n+1}$ are
linearly independent.

\subsection{Main results}

\subsubsection{Affine setting}

The following theorem was shown (in a slightly more general setting
of cones) in \cite{ArtsteinSlomka2011}. It roughly states that if
parallel lines in $n+1$ generic directions are mapped to parallel
lines, then the mapping is affine-additive.
\begin{thm}
\label{thm:newFTAG.Parallelism-n+1dir} Let $m\ge n\geq2$. Let $v_{1},v_{2},...,v_{n},v_{n+1}\in\R^{n}$
be $n$-independent and let $F:\R^{n}\rightarrow\R^{m}$ be an injection
that maps each line in $\L(v_{1},v_{2},...,v_{n},v_{n+1})$ onto a
line. Assume that parallel lines in this family are mapped onto parallel
lines. Then $F$ is affine-additive. Moreover, there exists two sets
of linearly independent vectors, $u_{1},\dots,u_{n}\in\R^{n}$ and
$w_{1},\dots,w_{n}\in\R^{m}$, and an additive bijective function
$f:\R\rightarrow\R$ with $f\left(1\right)=1$, such that for every
$x=\sum_{i=1}^{n}\alpha_{i}u_{i}$, 
\[
F(x)-F\left(0\right)=\sum_{i=1}^{n}f(\alpha_{i})w_{i}.
\]

\end{thm}
Let us remark about the main assumption in Theorem \ref{thm:newFTAG.Parallelism-n+1dir}
and the other results that follow. We mainly deal with injective mappings
that map some lines \textit{onto} lines. In some cases, we conclude
that the mapping is affine-additive. If continuity and surjectivity
assumptions are added to any of the statements in this section, then
the condition that lines are mapped \textit{onto} lines may be relaxed
and replaced by the condition that lines are mapped \textit{into}
lines (often referred to as ``collinearity''). This is due to Proposition
\ref{prop:Continuity_affine}. Furthermore, continuity and affine-additivity
implies affine-linearity. 

The proof of the above theorem relies upon the following theorem,
which states that if the parallelism condition is assumed for $n$
linearly independent directions, then the mapping must be of a diagonal
form.
\begin{thm}
\label{thm:newFTAG.Parallelism} Let $m\ge n\geq2$. Let $v_{1},v_{2},...,v_{n}\in\R^{n}$
be linearly independent and let $F:\R^{n}\rightarrow\R^{m}$ be an
injection that maps each line in $\L(v_{1},v_{2},...,v_{n})$ onto
a line, and moreover, that parallel lines in this family are mapped
onto parallel lines. Then, there exist two sets of linearly independent
vectors, $u_{1},\dots,u_{n}\in\R^{n}$ and $w_{1},\dots,w_{n}\in\R^{m}$,
and bijective functions $f_{1},...,f_{n}:\R\rightarrow\R$ with $f_{i}\left(0\right)=0$,
$f_{i}\left(1\right)=1$, such that for every  $x=\sum_{i=1}^{n}\alpha_{i}u_{i}$,
\[
F(x)-F\left(0\right)=\sum_{i=1}^{n}f_{i}(\alpha_{i})w_{i}.
\]
 In fact, one may choose $u_{i}=v_{i}$ for $i=1,\dots,n$.
\end{thm}
\noindent In this paper we investigate the case in which a parallelism
condition is \textit{not} assumed. However, since in the projective
case we shall use Theorems \ref{thm:newFTAG.Parallelism-n+1dir}-\ref{thm:newFTAG.Parallelism},
we provide their proofs in Appendix \ref{sec:paral}. 

Our first result is that when all lines in $n$ linearly independent
directions are mapped onto lines, the mapping must be of a very restricted
polynomial form:
\begin{thm}
\label{thm:poly-form} Let $m,n\geq2$. Let $v_{1},\dots,v_{n}\in\R^{n}$
be linearly independent and let $F:\R^{n}\to\R^{m}$ be an injection
that maps each line in $\L(v_{1},v_{2},...,v_{n})$ onto a line. Then
there exists a basis $u_{1},\dots,u_{n}$ in $\R^{n}$ such that for
every $x=\sum_{i=1}^{n}\alpha_{i}u_{i}$, 
\begin{equation}
F\left(x\right)=\sum_{\delta\in\{0,1\}^{n}}u_{\delta}\prod_{i=1}^{n}f_{i}^{\delta_{i}}(\alpha_{i})\label{eq:n-web-rep}
\end{equation}
where $u_{\delta}\in\R^{m}$, $\delta\in\{0,1\}^{n}$ and $f_{1},f_{2},...,f_{n}:\R\rightarrow\R$
are bijective with $f_{i}(0)=0$ and $f_{i}(1)=1$ for $i=1,...,n$.
Moreover, $m\ge n$, and if $m=n$ then $F$ is a bijection. 
\end{thm}
Adding one more direction in general position, in which lines are
mapped onto lines, yields further significant restrictions on the
polynomial form which is already implied by Theorem \ref{thm:poly-form}.
In order to state the theorem, we need to introduce further notation;
for any $\delta\in\{0,1\}^{n}$ denote $|\delta|=\sum_{i=1}^{n}\delta_{i}$.
\begin{thm}
\label{thm:Poly_n+1}Let $m,n\geq2$. Let $v_{1},\dots v_{n+1}\in\R^{n}$
be in general position. Let $F:\R^{n}\to\R^{m}$ be an injective mapping
that maps each line in $\mathcal{L}\left(v_{1},\dots,v_{n+1}\right)$
onto a line. Then there exists a basis $u_{1},\dots,u_{n}$ in $\R^{n}$
such that for every $x=\sum_{i=1}^{n}\alpha_{i}u_{i}$, 
\begin{align}
F\left(x\right)=\sum_{\delta\in\{0,1\}^{n}}u_{\delta}\prod_{i=1}^{n}f_{i}^{\delta_{i}}(\alpha_{i})\label{eq:poly_n+1_form}
\end{align}
where $f_{1},f_{2},...,f_{n}:\R\rightarrow\R$ are additive bijections
with $f_{i}(1)=1$ for $i=1,...,n$, and $u_{\delta}\in\R^{m}$ satisfy
the following conditions: 
\begin{itemize}
\item $u_{\delta}=0$ for all $\delta$ with $|\delta|\ge\frac{n+2}{2}$,
and 
\item for each $2\le k<\frac{n+2}{2}$ and every $0\le l\le k-2$ indices
$1\le i_{1}<\cdots<i_{l}\le n$, 
\[
\sum_{\substack{|\delta|=k,\\
\delta_{i_{1}}=\cdots=\delta_{i_{l}}=1
}
}u_{\delta}=0.
\]

\end{itemize}
\noindent Moreover $m\ge n$, and if $m=n$ then $F$ is a bijection.
Conversely, any mapping $F$ of the form as in the right hand side
of \eqref{eq:poly_n+1_form}, which satisfy the given conditions on
the coefficients $u_{\delta}$, takes each line in $\L(e_{1},...,e_{n},v)$,
where $v=\sum_{i=1}^{n}e_{i}$,  onto a line.\end{thm}
\begin{rem}
\label{rem:compose-A}The basis $u_{1},\dots,u_{n}$ which appears
in the statement of Theorem \ref{thm:Poly_n+1} is the one satisfying
$u_{i}=\lambda_{i}v_{i}$, where $\lambda_{i}\in\R$ for which $\sum u_{i}=v_{n+1}$.
 Similarly, the basis which appears in Theorem \ref{thm:newFTAG.Parallelism}
or in Theorem \ref{thm:poly-form} is simply $u_{i}=v_{i}$ for each
$i$.
\end{rem}
Theorem \ref{thm:Poly_n+1} is sharp in the sense that one may construct
an injective polynomial map of degree $\left\lceil \frac{n-1}{2}\right\rceil $
which satisfies the assumptions of the theorem. In particular, the
bound $\left\lceil \frac{n}{2}\right\rceil $ on the degree of the
given polynomial form is optimal for even dimensions $n$. This fact
is explained in Example \ref{rem:n+1:Sharpness}.\\

Theorem \ref{thm:Poly_n+1} may be used to derive various generalizations
of the classical fundamental theorem of affine geometry, where collinearity
preservation is assumed only for a finite number of directions of
lines. For the three dimensional case we prove the following surprisingly
strong generalization. 
\begin{thm}
\label{thm:NFTAG3D} Let $v_{1},v_{2},\dots,v_{5}\in\R^{3}$ be $3$-independent.
Let $F:\R^{3}\to\R^{3}$ be an injective mapping that maps each line
in $\mathcal{L}\left(v_{1},\dots,v_{5}\right)$ onto a line. Then
$F$ is affine-additive. Moreover, there exist a basis $u_{1},u_{2},u_{3}\in\R^{3}$,
a basis $w_{1},w_{2},w_{3}\in\R^{3}$, and additive bijections $f_{1},f_{2},f_{3}:\R\to\R$
with $f_{i}\left(1\right)=1$, such that for every $x=\sum_{i=1}^{3}\alpha_{i}u_{i}$,
\[
F\left(x\right)-F\left(0\right)=\sum_{i=1}^{3}f_{i}\left(\alpha_{i}\right)w_{i}.
\]

\end{thm}

\subsubsection{Projective setting}

Let us state two new versions for the fundamental theorem of projective
geometry. In both theorems, projective lines passing through $n+2$
different points are assumed to be mapped onto lines. In the first
theorem all points are assumed to be in general position:
\begin{thm}
\label{thm:Pr+2} Let $n\geq2$. Let $\bar{p}_{1},...,\bar{p}_{n},\bar{p}_{n+1}\in\RP^{n}$
be generic and let $\bar{p}_{n+2}\in\RP^{n}$ be a projective point
satisfying $\bar{p}_{n+2}\not\in\Psp\{\bar{p}_{1},...,\bar{p}_{n}\}$
and also $\bar{p}_{n+2}\neq\bar{p}_{n+1}$. Let $F:\RP^{n}\to\RP^{n}$
be an injective mapping that maps any projective line containing one
of the points $\bar{p}_{1},...,\bar{p}_{n+2}$ onto a projective line.
Then $F$ is a projective-linear mapping.
\end{thm}
In the second theorem, $n+1$ points are assumed to be contained in
a projective subspace of co-dimension $1$, where the $(n+2)^{th}$
direction lies outside the subspace:
\begin{thm}
\label{thm:Pr+1} Let $n\geq2$. Let $\bar{H}\subset\RP^{n}$ be a
projective subspace of co-dimension $1$, and let $\bar{p}_{1},\bar{p}_{2}...,\bar{p}_{n+1}\in\bar{H}$
be generic in $\bar{H}$. Let $\bar{p}_{n+2}\in\R P^{n}\setminus\bar{H}$.
Let $F:\RP^{n}\to\RP^{n}$ be an injective mapping that maps any projective
line containing one of the points $\bar{p}_{1},...,\bar{p}_{n+2}$
onto a projective line. Then $F$ is a projective-linear mapping.
\end{thm}
As in the affine setting, if a continuity assumption is added to Theorems
\ref{thm:Pr+2}-\ref{thm:Pr+1}, along with the assumption that the
mapping is surjective, the assumption that projective lines are mapped
\textit{onto} projective lines may be replaced by a \textit{collinearity}
assumption. This is due to Proposition \ref{prop:Continuity_proj}.
Moreover, if a continuity assumption is added to Theorem \ref{thm:Pr+1},
one may easily verify that the assumption on lines through $\bar{p}_{n+2}$
may be removed (by a minor adjustment of its proof). However, in Theorem
\ref{thm:Pr+2}, this is not possible.

\subsubsection{Other number fields}

The algebraic nature of our proofs in this paper implies that many
of our results hold for fields other than $\R$. For example, the
results hold for $\Z_{p}$, with $p\neq2$. However, for simplicity
of the exposition, we focus solely on $\R$, which keeps our arguments
clearer to the reader.

\subsection*{Acknowledgments}

We wish to thank Prof. Leonid Polterovitch for useful remarks.

\section{Fundamental theorems of affine geometry\label{sec:FTAG}}

\subsection{Introductory remarks}

Let us begin by an example which shows that the most straightforward
generalization, which works in the projective setting, does not hold
in the affine setting. Namely, we can find a map (actually, a polynomial
automorphism) $P:\R^{3}\to\R^{3}$ which maps all lines in four directions,
each three of which are linearly independent, to lines, and yet is
non-linear. 
\begin{example}
\label{exm:R3} Define $P:\R^{3}\to\R^{3}$ by  
\[
P(x_{1},x_{2},x_{3})=(x_{1}+x_{3}(x_{1}-x_{2}),x_{2}+x_{3}(x_{1}-x_{2}),x_{3}).
\]
Then clearly $\{P(w+te_{i})\}_{t\in\R}$ is a line for any $i\in\{1,2,3\}$
and $w\in\R^{3}$. One can also check that $\left\{ P(v+t(e_{1}+e_{2}-e_{3}))\right\} _{t\in\R}$
is parallel to $(1+w_{2}-w_{1},1+w_{2}-w_{1},-1)$.
\end{example}
Section \ref{sec:FTAG} is organized as follows. In Section \ref{sec:preliminaries}
we gather some basic useful facts. In Section \ref{sec:R2ToRn} we
consider the plane $\R^{2}$, and see what must be the form of a mapping
$F:\R^{2}\to\R^{m}$ which maps all lines in two directions onto lines
in $\R^{n}$ for $n\geq2$ . This is given as Theorem \ref{thm:General2Dcoll}
below. Then, In Section \ref{sec: Poly-form}, we use this as an induction
basis for the general form of a mapping $F:\R^{n}\to\R^{m}$ which
maps all lines in a family $\L(v_{1},\ldots,v_{n})$ onto lines. This
form was given in Theorem \ref{thm:poly-form} in the introduction.
In Section \ref{sec: n+1-directions}, we see how an additional $(n+1)^{th}$
direction, for which lines are mapped onto lines, further restricts
the polynomial form of the mapping, obtaining Theorem \ref{thm:Poly_n+1}.
In Section \ref{sub:affine-additive} we discuss several cases in
which collinearity for lines in a finite number of directions suffices
to derive affine-additivity, in particular Theorem \ref{thm:NFTAG3D}.

\subsection{Preliminary facts and results\label{sec:preliminaries}}

We will use the notation $\sp\{v_{1},\ldots,v_{k}\}$ to denote the
linear span of the $k$ vectors $\{v_{i}\}_{i=1}^{k}$, so a line
$a+\R b$ can also be written as $a+\sp\{b\}$.

Let $F:\R^{n}\to\R^{m}$ be an injective mappings which maps lines
in a given family $\L(v_{1},...,v_{k})$ onto lines. We will have
use of the following simple self-evident facts, concerning such a
mapping, which we gather here  for future reference:
\begin{fact}
\label{fact:trans} A translation of $F$ by any vector $v_{0}\in\R^{m}$,
$F(x)+v_{0}$ is also an injection which maps each line in $\L(v_{1},...,v_{k})$
onto a line. We will usually use this property to assume without loss
of generality that $F(0)=0$.
\end{fact}

\begin{fact}
\label{fact:linear}Compositions of $F$ with invertible linear transformations
$B\in{\rm GL}_{n}(\R)$ and $A\in{\rm GL}_{m}(\R)$ are also injective,
and satisfy the property of line-onto-line for certain lines: $A\circ F$
maps each line in $\L(v_{1},...,v_{k})$ onto a line, whereas $F\circ B$
maps each line in $\L(B^{-1}v_{1},B^{-1}v_{2},...,B^{-1}v_{k})$ onto
a line. Moreover, for two sets of $n+1$ points in general positions,
$v_{1},\dots,v_{n+1}\in\R^{n}$ and $u_{1},\dots u_{n+1}\in\R^{n}$,
there exists an invertible linear transformation $B\in\GL_{n}\left(\R\right)$
such that $F\circ B$ maps each line in $\L\left(u_{1},\dots,u_{n+1}\right)$
onto a line (this is also an easy consequence of Theorem \ref{thm:PrLiPo}
below). We shall often use this fact to assume without loss of generality
that we are working with some standard families of lines. 
\end{fact}

\begin{fact}
\label{fact:dimensions} The image of an injective mapping $F:\R^{2}\to\R^{m}$,
which maps each line in a given family $\L(v_{1},v_{2})$ (where $v_{1},v_{2}$
are linearly independent) is always contained in a three-dimensional
affine subspace of $\R^{m}$, and so by composing $F$ with an appropriate
linear transformation we may identify this subspace with $\R^{3}$.
In general, for any $n$, and $F:\R^{n}\to\R^{m}$ (injective that
maps lines-onto-lines for a given family $\L(v_{1},...,v_{k})$),
one may show that the image of $F$ is contained in an affine subspace
of dimension $2^{n-1}$. 
\end{fact}

\begin{fact}
From any point in the image of $F$, there emanate exactly $k$ lines
which are the images of lines in $\L(v_{1},...,v_{k})$ under $F$,
intersecting only at that point.
\end{fact}

\subsection{Maps from $\protect\R^{2}$ to $\protect\R^{n}$}

\label{sec:R2ToRn} One may verify  that a bijection which maps lines
onto lines, in all directions, also maps planes onto planes. This
fact implies that the classical fundamental theorem of affine geometry
is essentially a two dimensional claim (since linearity is also a
two dimensional notion). Since in the plane any two non-intersecting
lines must be parallel, collineations from the plane to itself have
a very restrictive ``diagonal'' form, even if the line-to-line condition
is assumed only for all lines in two pre-chosen directions, as the
following lemma suggests. Recall that $\L(e_{1},e_{2})$ denotes the
family of all lines in directions $e_{1}$ or $e_{2}$.
\begin{lem}
\label{lem:PlaneToPlane2directionsColl} Let $F:\R^{2}\to\R^{2}$
be an injection that maps each line in $\L(e_{1},e_{2})$ onto a line.
Then $F$ is of the following form 
\[
F(se_{1}+te_{2})-F\left(0\right)=f\left(s\right)u_{1}+f\left(t\right)u_{2},
\]
where $u_{1},u_{2}\in\R^{2}$ are linearly independent, and $f,g:\R\to\R$
are bijections with $f(0)=g(0)=0$, and $f(1)=g(1)=1$.\end{lem}
\begin{proof}
By translating $F$ we may assume that $F(0)=0$. Moreover, by composing
$F$ with a linear transformation from the left, we may assume that
$F(e_{1})=e_{1}$ and $F(e_{2})=e_{2}$. This linear transformation
must be invertible since $e_{1}$ and $e_{2}$ cannot be mapped by
the original $F$ to linearly dependent vectors, as (together with
the line-onto-line) this would contradict the injectivity of $F$.
Define $f:\R\to\R$ by the relation $F(se_{1})=f(s)e_{1}$. Similarly,
define $g:\R\to\R$ by $F(te_{2})=g(t)e_{2}$. Clearly, by our assumptions,
$f$ and $g$ are bijections satisfying $f(0)=g(0)=0$ and $f(1)=g(1)=1$.

Since $F$ is injective and since any two parallel lines in $\L(e_{1},e_{2})$
are mapped onto two lines, they must not intersect and so must be
parallel. Hence, for $i=1,2$ and for every line $l$ in $\L(e_{i})$,
$F(l)$ is again in $\L(e_{i})$. Now, any point $s_{0}e_{1}+t_{0}e_{2}$
is the intersection of the the lines $s_{0}e_{1}+\sp\{e_{2}\}$ and
$\sp\{e_{1}\}+t_{0}e_{2}$ which are mapped to $f(s_{0})e_{1}+\sp\{e_{2}\}$
and to $\sp\{e_{1}\}+g(t_{0})e_{2}$, respectively. Since their images
intersect at $F(s_{0}e_{1}+t_{0}e_{2})$, we conclude that $F(s_{0}e_{1}+t_{0}e_{2})=f(s_{0})e_{1}+g(t_{0})e_{2}$,
as required. 
\end{proof}
If two parallel lines in $\R^{2}$ are mapped to lines by an injection,
then their images do not intersect. Above, as the image of $F$ was
contained in a plane, this meant that the images were parallel lines.
If the images are in $\R^{n}$ with $n>2$, this need no longer be
the case, and the images of parallel lines can be skew. Indeed, one
easily constructs collineations in two directions, embedding the plane
into $\R^{n}$ with $n>2$, the image of which does not contain {\em
any} pair of parallel lines. (This image is, however, contained in
a three dimensional affine subspace, see Fact \ref{fact:dimensions}).
Still, it turns out that such mappings also have a specific simple
form. 
We prove the following theorem, which is strongly connected with a
known theorem about determination of doubly ruled surfaces, see Remark
\ref{rem:doubly}.
\begin{thm}
\label{thm:General2Dcoll} Let $n\geq2$. Let $F:\R^{2}\rightarrow\R^{n}$
be an injective mapping that maps each line in $\L(e_{1},e_{2})$
onto a line. Then $F$ is given by 
\begin{equation}
F(se_{1}+te_{2})-F\left(0\right)=f\left(s\right)u_{1}+g\left(t\right)u_{2}+f\left(s\right)g\left(t\right)u_{3}\label{eq:2to-n-2directions}
\end{equation}
where $u_{1},u_{2}\in\R^{n}$ are linearly independent, $u_{3}\in\R^{n}$,
and $f,g:\R\to\R$ are bijective with $f(0)=g(0)=0$ and $f(1)=g(1)=1$.\end{thm}
\begin{proof}
Again, by translating $F$ we may assume that $F(0)=0\in\R^{n}$.
Also, by considering $A^{-1}\circ F$ for an invertible linear $A$
we may assume that the image of $F$ is contained in $\sp\{e_{1},e_{2},e_{3}\}$
which we identify with $\R^{3}$ (see Fact \ref{fact:dimensions}).

It will be useful to denote by $\L_{F}(e_{i})=\{F(l):l\in\L(e_{i})\}$
the family of lines which are images of the lines in $\L(e_{i})$
under $F$. Note that for each $i$ this is a non-intersecting family
of lines whose union is the image of $F$. Assume first, that there
exist two different lines $l_{1}$, $l_{2}$ in $\L(e_{1},e_{2})$
that $F$ maps to parallel lines. Since they do not intersect, they
both belong to one family $\L(e_{i})$, so assume without loss of
generality that $l_{1}$ and $l_{2}$ are in $\L(e_{1})$. Then, any
line in $\L_{F}(e_{2})$ must intersect both $F(l_{1})$ and $F(l_{2})$,
and so they all lie on one affine plane. Since $\L_{F}(e_{2})$ is
the image of $F$, it follows that it is contained in a two dimensional
affine subspace. In Lemma \ref{lem:PlaneToPlane2directionsColl},
it was shown that in such a case $F$ is of the form (\ref{eq:2to-n-2directions})
with $u_{3}=0$.

We move on to the second case, in which no two lines in $\L_{F}(e_{1},e_{2})$
are parallel. This implies that $F(e_{1}+e_{2})\not\in\sp\{F(e_{1}),F(e_{2})\}$,
and so we may pick an invertible linear transformation $B\in GL_{3}(\R)$
such that $BF(e_{2})=e_{2}$ $BF(e_{1})\in\sp\{e_{1}\}$, and $BF(e_{1}+e_{2})\in e_{2}+\sp\{e_{1}+e_{3}\}$.
Then 
\[
BF(\sp\{e_{1}\})=\sp\{e_{1}\},\;BF(\sp\{e_{2}\})=\sp\{e_{2}\},
\]
and 
\[
BF(e_{2}+\sp\{e_{1}\})=e_{2}+\sp\{e_{1}+e_{3}\}.
\]
Without loss of generality we assume that $F$ itself satisfies the
above.

Choose any $a\neq0,1$ in $\R$. As $ae_{2}$ is in the image of $F$,
we may consider the line $l\in\L_{F}(e_{1})$ emanating from it. Denote
its corresponding line, $F^{-1}(l)\in\L(e_{1})$, by $l'$. Since
every point on the line $l'$ intersects a line, parallel to $e_{2}$,
connecting a point in $\sp\{e_{1}\}$ and a point in $e_{2}+\sp\{e_{1}\}$,
then also (after applying $F$) any point of $l$ lies on a line connecting
a point in $\sp\{e_{1}\}$ and a point in $e_{2}+\sp\{e_{1}+e_{3}\}$.
The union of all such lines includes all of $\R^{3}$ except for two
parallel planes (without the two relevant lines): the $XZ$-plane
and its translation by $e_{2}$. Indeed, lines connecting $\sp\{e_{1}\}$
and $e_{2}+\sp\{e_{1}+e_{3}\}$ consist of points of the form 
\begin{equation}
\lambda(xe_{1})+(1-\lambda)(ye_{1}+e_{2}+ye_{3}),\;\lambda,x,y\in\R,\label{eq:middleLine}
\end{equation}
and so do not include points of the form $(xe_{1}+ze_{3})$ with $z\neq0$
nor points of the form $(xe_{1}+e_{2}+ze_{3})$ with $z\neq x$. Since
$l$ is a line of the form 
\[
\{ae_{2}+t(a_{1}e_{2}+a_{2}e_{2}+a_{3}e_{3}):\;t\in\R\}
\]
it follows that $a_{2}=0$, $a_{3}\neq0$ and $a_{1}\neq a_{3}$,
so that it will neither intersect these planes nor be parallel to
one of the lines $\sp\{e_{1}\},\sp\{e_{1}+e_{3}\}$. 

By composing $F$ from the left with the invertible linear transformation
{[}in coordinates of $\R^{3}$ corresponding to $x=\left(x_{1},x_{2},x_{3}\right)^{T}=x_{1}e_{1}+x_{2}e_{2}+x_{3}e_{3}${]}
\[
{\displaystyle A=\begin{pmatrix}\frac{1-a}{a_{1}-a_{3}} &  & 0 &  & \frac{a}{a_{3}}-\frac{1-a}{a_{1}-a_{3}}\\
\\
0 &  & 1 &  & 0\\
\\
0 &  & 0 &  & \frac{a}{a_{3}}
\end{pmatrix}}
\]
we may assume without loss of generality that $(a_{1}e_{1}+a_{3}e_{3})=(e_{1}+ae_{3})$
(as still $A(\sp\{e_{1}\})=\sp\{e_{1}\}$, $Ae_{2}=e_{2}$, $A(\sp\{e_{1}+e_{3}\})=\sp\{e_{1}+e_{3}\}$,
and $A(a_{1}e_{1}+a_{3}e_{3})=(e_{1}+ae_{3})$, and so all of our
assumptions so far still hold). Furthermore, we may also assume without
loss of generality that $Fe_{1}=e_{1}$ 
by composing with an additional diagonal matrix of the form 
\[
D=\begin{pmatrix}\alpha\\
 & 1\\
 &  & \alpha
\end{pmatrix}
\]
where the coefficient $\alpha\in\R$ is defined by $F(e_{1})=\alpha^{-1}e_{1}$.

Summarizing the above, we have $F(0)=0$, $F(e_{1})=e_{1}$, $F(e_{2})=e_{2}$,
$F(e_{2}+\sp\{e_{1}\})=e_{2}+\sp\{e_{1}+e_{3}\}$ and $F(l)=ae_{2}+\sp\{e_{1}+ae_{3}\}$.
Let us check which lines belong to $\L_{F}(e_{2})$. Since each one
of them intersects the three aforementioned lines and since we have
that 
\[
ae_{2}+t(e_{1}+ae_{3})=(1-a)(te_{1})+a(te_{1}+e_{2}+te_{3}),
\]
it follows that from every point $te_{1}$, emanates the line $\{te_{1}+s(e_{2}+te_{3})\}_{s\in\R}$
in $\L_{F}(e_{2})$. In other words, the image of the map $F$ consists
of points of the form 
\[
\{te_{1}+se_{2}+tse_{3}:t\in\R,s\in\R\}.
\]
It is easily checked that from every point on this surface there emanate
exactly two straight lines, the line $\{te_{1}+s(e_{2}+te_{3})\}_{s\in\R}$
which we saw is in $\L_{F}(e_{2})$, and the line $\{se_{2}+t(e_{1}+se_{3})\}_{t\in\R}$
which must thus belong to $\L_{F}(e_{1})$.


Since $F$ is injective and maps lines in $\L(e_{1},e_{2})$ onto
lines, there exist bijective functions $f,g:\R\rightarrow\R$ satisfying
the equations $F(xe_{1})=f(x)e_{1}$ and $F(xe_{2})=g(x)e_{2}$ for
all $x\in\R$. By our assumptions, we have $f(0)=g(0)=0$ and $f(1)=g(1)=1$. 

We conclude that for any fixed $t_{0},s_{0}\in\R$, we have for all
$s\in\R$
\[
F(t_{0},s)\in f(t_{0})e_{1}+\sp\{e_{2}+f(t_{0})e_{3}\},
\]
and for all $t\in\R$ 
\[
F(t,s_{0})\in g(s_{0})e_{2}+\sp\{e_{1}+g(s_{0})e_{3}\}.
\]
Since the intersection of any two lines in $\L(e_{1},e_{2})$ is mapped
to the intersection of the their images, we get that 
\[
F(t_{0},s_{0})=f(t_{0})e_{1}+g(s_{0})e_{2}+f(t_{0})g(s_{0})e_{3}
\]
as required. \end{proof}
\begin{rem}
As in Lemma \ref{lem:PlaneToPlane2directionsColl}, we note that by
the same reasoning, Theorem \ref{thm:General2Dcoll} holds for any
field $\F\neq\Z_{2}$. In $\Z_{2}$, the only reason the proof does
not hold is that we could not have chosen an element $a\neq0,1$.
In other words, the proof requires at least three parallel lines in
each family, which in the case of $\Z_{2}$ do not exist.
\end{rem}

\begin{rem}
\label{rem:doubly} Note that the image of $F$ is no other than a
linear image of the well known hyperbolic-paraboloid $\{(x,y,xy)^{T}:x\in\R,y\in\R\}$.
In terms of surfaces, it is known that up to a linear image, only
two non-planar surfaces exist in $\R^{3}$ which are ``doubly ruled'',
which means that they have two essentially different parameterizations
as a disjoint union of lines. One of these surfaces is the hyperbolic-paraboloid,
and the second is the rotational hyperboloid, which can be parameterized,
say, as $\{\cos(s)-t\sin(s),\sin(s)+t\cos(s),t)^{T}\}$. While the
image of an $F$ satisfying the conditions of Theorem \ref{thm:General2Dcoll}
is automatically doubly ruled, it is not true in general, as the last
example shows, that every doubly ruled surface can be parameterized
in such a way that gives lines whenever one of the parameters is kept
constant.

Theorem \ref{thm:General2Dcoll} can be deduced in a relatively simple
manner from the so-called ``determination of doubly ruled surfaces
in $\R^{3}$''. For a proof of this theorem see for example \cite{Kri99}.
Still, we chose to give the direct proof above to make the exposition
self contained.

We remark that Alexandrov in his proof of the fact that the only isomorphisms
which preserve light-cone structure invariant are affine, used the
aforementioned characterization in a similar way to that in which
the authors proved a cone-isomorphism result in \cite{ArtsteinSlomka2011}.


In the context of doubly ruled surfaces, it makes sense to ask, for
example, whether $\R^{3}$ can be parameterized in a non-linear way
so that it is ``triply ruled'', (the answer being ``yes'', simply
take 
\[
F(s,t,r)=\left(\begin{array}{c}
s\\
t\\
st-r
\end{array}\right)
\]
which spans all of $\R^{3}$ bijectively). ``Triply ruled'' here
can mean only the more restricted definition namely that there is
a parametrization in which, fixing any two parameters, the third one
induces a line, since clearly infinitely many different parameterizations
of $\R^{3}$ as a union of lines exist.

When speaking of $n$-ruled surfaces in $\R^{m}$ for general $n,m$,
one can also use the less restrictive definitions, either of there
being $n$ different (possibly generic) lines through every point
in the surface, or of the existence of $n$ essentially different
parameterizations of the surface as a union of lines. To the best
of our knowledge, such surfaces have not been characterized for $n\neq2$.
In the next section we show that these ``parameterized $n$-ruled
surfaces'', that is, images of bijective ``$\L(\{e_{i}\}_{i=1}^{n})$-collineations''
of $\R^{n}$ into $\R^{m}$, must be of a very restrictive polynomial
form.
\end{rem}

\subsection{Collineations in $n$ directions - A polynomial form}

\label{sec: Poly-form}

In this section, we prove Theorem \ref{thm:poly-form}. We show that
the special form of plane collineations which was given in Theorem
\ref{thm:General2Dcoll}, carries over to higher dimensions by induction.
Note that indeed, the given form in Theorem \ref{thm:poly-form} is
precisely the form given in Theorem \ref{thm:General2Dcoll} where
$u_{1}=u_{(1,0)},u_{2}=u_{(0,1)},u_{3}=u_{(1,1)}$. We shall not provide
sufficient conditions which ensure the map is injective. 
\begin{proof}
[Proof of Theorem \ref{thm:poly-form}] By composing $F$ with the
linear transformationt taking each $e_{i}$ to $v_{i}$ we may assume
without loss of generality that $v_{i}=e_{i}$. It will be convenient
for us to use coordinates in the following. Namely, $F\left(x_{1},x_{2},\dots,x_{n}\right):=F\left(\sum_{i=1}^{n}x_{i}e_{i}\right)$.
The proof goes by induction on $n$, the case $n=2$ settled already
in Theorem \ref{thm:General2Dcoll}. We assume our claim holds for
$(n-1)$ and prove it for $n$. The induction hypothesis, applied
for the function of $(n-1)$ variables $F(\cdot,\ldots,\cdot,x_{j},\cdot,\ldots,\cdot)$
with $j^{th}$ coordinate fixed to be equal $x_{j}$ implies that
\begin{equation}
F(x_{1},x_{2},...,x_{n})=\sum_{\underset{\delta_{j}=0}{\delta\in\{0,1\}^{n}}}u_{\delta}(x_{j})\prod_{k\neq j}f_{k,x_{j}}^{\delta_{k}}(x_{k}).\label{eq:xj-fixed}
\end{equation}
A-priori, the bijections $f_{k,x_{j}}$ can depend on the value of
$x_{j}$, as may the vector coefficients $u_{\delta}$, including
the coefficient $u_{0}(x_{j})=F(x_{j}e_{j})$ (here $0=(0,\ldots,0)$).

Our aim is to first show that $f_{k,x_{j}}$ does not depend on $x_{j}$,
and then to show that there is another bijection $f_{j}:\R\to\R$,
such that all the coefficients $u_{\delta}$ depends in an affine
way on $f_{j}(x_{j})$. This would complete the proof.

Similarly to (\ref{eq:xj-fixed}), fixing a different variable $x_{i}$
($i\neq j$) we may write $F$ as 
\begin{equation}
F(x_{1},x_{2},...,x_{n})=\sum_{\underset{\delta_{i}=0}{\delta\in\{0,1\}^{n}}}v_{\delta}(x_{i})\prod_{k\neq i}g_{k,x_{i}}^{\delta_{k}}(x_{k}).\label{eq:xi-fixed}
\end{equation}
Let us introduce some index-simplifying notation: let $\delta^{j}\in\{0,1\}^{n}$
denote the vector with value $1$ in the $j^{th}$ entry, and value
$0$ in all other entries, and for $i\neq j$, let $\delta^{i,j}=\delta^{i}+\delta^{j}$.
Denote 
\[
F_{i,j}(a,b)=F(ae_{j}+be_{i}).
\]

Computing $F(x_{j}e_{j})$ using the two representations (\ref{eq:xj-fixed})
and (\ref{eq:xi-fixed}) 
we get 
\begin{equation}
v_{0}(0)+v_{\delta^{j}}(0)g_{j,0}(x_{j})=u_{0}(x_{j})\label{eq:u-delta-0}
\end{equation}
and similarly, computing $F(x_{i}e_{i})$ we get 
\begin{equation}
u_{0}(0)+u_{\delta^{i}}(0)f_{i,0}(x_{i})=v_{0}(x_{i}).\label{eq:d-delta-0}
\end{equation}
Let $p\neq i,j$ be any other index. From the two representations
of $F(x_{p}e_{p})$ we have 
\begin{equation}
u_{0}(0)+u_{\delta^{p}}(0)f_{p,0}(x_{p})=v_{0}(0)+v_{\delta^{p}}(0)g_{p,0}(x_{p})\label{eq:fgp0}
\end{equation}
so that by setting $x_{p}=1$ and using that $u_{0}(0)=v_{0}(0)=F(0,0,...,0)$
we get $v_{\delta^{p}}(0)=u_{\delta^{p}}(0)$. Note that $u_{\delta^{p}}(0)\neq0$,
otherwise $F(x_{p}e_{p})$ would be independent of $x_{p}$, which
is a contradiction to the injectivity of $F$ (similarly, $v_{\delta^{p}}(0)\neq0$).
Using equation (\ref{eq:fgp0}) once more (subtracting $F(0)$ and
canceling $u_{\delta^{p}}(0)$) we get that $f_{p,0}(x_{p})=g_{p,0}(x_{p})$,
which holds for all $p\neq i,j$. Writing $F_{p,i}(x_{p},x_{i})$
in our two forms, we get 
\begin{align}
F_{p,i}(x_{p},x_{i})=u_{0}(0)+u_{\delta^{i}}(0)f_{i,0}(x_{i})+u_{\delta^{p}}(0)f_{p,0}(x_{p})+u_{\delta^{i,p}}(0)f_{i,0}(x_{i})f_{p,0}(x_{p})\label{eq:Fip-xj}
\end{align}
and 
\begin{equation}
F_{p,i}(x_{p},x_{i})=v_{0}(x_{i})+v_{\delta^{p}}(x_{i})g_{p,x_{i}}(x_{p})\underset{(\ref{eq:d-delta-0})}{=}u_{0}(0)+u_{\delta^{i}}(0)f_{i,0}(x_{i})+v_{\delta^{p}}(x_{i})g_{p,x_{i}}(x_{p})\label{eq:Fip-xi}
\end{equation}
Comparing these two equations yields 
\begin{equation}
v_{\delta^{p}}(x_{i})g_{p,x_{i}}(x_{p})=u_{\delta^{p}}(0)f_{p,0}(x_{p})+u_{\delta^{i,p}}(0)f_{i,0}(x_{i})f_{p,0}(x_{p})\label{eq:d-delta-p}
\end{equation}
and by plugging in $x_{p}=1$ we get 
\begin{equation}
v_{\delta^{p}}(x_{i})=u_{\delta^{p}}(0)+u_{\delta^{i,p}}(0)f_{i,0}(x_{i})\label{eq:d-delta-p-xi}
\end{equation}
and we already see that the dependence of $v_{\delta^{p}}(x_{i})$
on $x_{i}$ is affine-linear in $f_{i,0}(x_{i})$.

Rearranging equation (\ref{eq:d-delta-p}) we have 
\[
g_{p,x_{i}}(x_{p})v_{\delta^{p}}(x_{i})=f_{p,0}(x_{p})\left[u_{\delta^{p}}(0)+u_{\delta^{i,p}}(0)f_{i,0}(x_{i})\right]
\]
and plugging equation (\ref{eq:d-delta-p-xi}) into it (recall that
$v_{\delta^{p}}(x_{i})\neq0$) we get $g_{p,x_{i}}(x_{p})=f_{p,0}(x_{p})$
and in particular, $g_{p,x_{i}}$ is independent of $x_{i}$. Similarly,
we get $f_{p,x_{j}}(x_{p})=g_{p,0}(x_{p})$ and so $f_{p}:=f_{p,0}=f_{p,x_{j}}=g_{p,x_{i}}$
for every $p\neq i,j$.

Clearly the indices $i,j$ are not special, so one can repeat the
considerations and compare the first representation (for a fixed $x_{j}$)
with a different representation, for a fixed $x_{l}$ (with $l\neq i,j$).
In that case then we would get that also $f_{i,x_{j}}$ is independent
of $x_{j}$ (and similarly, $g_{j,x_{i}}$ is independent of $x_{i}$).
We denote $f_{i}:=f_{i,0}=f_{i,x_{j}}$ and $g_{j}:=g_{j,0}=g_{j,x_{i}}$.

So, going back to our two representations, we have 
\begin{equation}
F(x_{1},x_{2},...,x_{n})=\sum_{\underset{\delta_{j}=0}{\delta\in\{0,1\}^{n}}}u_{\delta}(x_{j})\left[\prod_{k\neq i,j}f_{k}^{\delta_{k}}(x_{k})\right]f_{i}^{\delta_{i}}(x_{i}),\label{eq:F-xj-independent}
\end{equation}
and 
\begin{equation}
F(x_{1},x_{2},...,x_{n})=\sum_{\underset{\delta_{i}=0}{\delta\in\{0,1\}^{n}}}v_{\delta}(x_{i})\left[\prod_{k\neq i,j}f_{k}^{\delta_{k}}(x_{k})\right]g_{j}^{\delta_{j}}(x_{j}).\label{eq:F-xi-independent}
\end{equation}

Next, we show that each coefficient $u_{\delta}(x_{j})$ in representation
(\ref{eq:F-xj-independent}) depends in an affine way on $g_{j}(x_{j})$,
that is 
\begin{equation}
u_{\delta}(x_{j})=w_{\delta}+y_{\delta}g_{j}(x_{j}).\label{eq:u_delta}
\end{equation}
for some $w_{\delta}$ and $y_{\delta}$. 
This is done by induction on the number of $``1"$ entries in $\delta$,
where the induction base is given in (\ref{eq:u-delta-0}). Assume
$u_{\delta}(x_{j})$ has the required form for $\delta$ with no more
than $N$ non-zero entries. Set $N+1$ coordinates $p_{1},p_{2},...,p_{N+1}$
(all different from $j$) and let $\hat{\delta}=\delta^{p_{1},p_{2},...,p_{N+1}}=\sum_{k=1}^{N+1}\delta^{p_{k}}$.
We will show that $u_{\hat{\delta}}(x_{j})$ depends in an affine
way on $g_{j}(x_{j})$. For $F(x_{j}e_{j}+\sum_{k=1}^{N+1}e_{p_{k}})$,
the representation given in (\ref{eq:F-xi-independent}) gives us
an expression of the form $w_{1}+w_{2}g_{j}(x_{j})$. Comparing with
the representation given in (\ref{eq:F-xj-independent}), we have

\[
\sum_{\substack{\delta\in\{0,1\}^{n},\\
\delta_{k}=0,\;\forall k\neq p_{1},\ldots p_{N+1}
}
}u_{\delta}(x_{j})=w_{1}+w_{2}g_{j}(x_{j})
\]
where the sum in the left hand side is over indices $\delta$ all
of which have at most $N$ non-zero entries except $\hat{\delta}$.
Rearranging terms and using the induction hypothesis we get 
\[
u_{\hat{\delta}}(x_{j})=\hat{v}_{1}+\hat{v}_{2}g_{j}(x_{j})
\]
for some vectors $\hat{v}_{1},\hat{v}_{2}\in\R^{n}$, as required.

Plugging equation (\ref{eq:u_delta}) into \eqref{eq:F-xj-independent}
and denoting $g_{j}=f_{j}$, $w_{\delta}=u_{\delta}$ and $y_{\delta}=u_{\delta+\delta^{j}}$,
we get the form of equation (\ref{eq:n-web-rep}) for dimension $n$,
as claimed.

Next we show that $m\ge n$. To this end, consider the injective polynomial
mapping 
\[
\widetilde{F}(x_{1},\dots,x_{n})=F(f_{1}^{-1}(x_{1}),\dots,f_{n}^{-1}(x_{n})).
\]
Such a map must satisfy that $m\ge n$. This follows, for example,
from a result of A. Bia\l ynicki-Birula and M. Rosenlicht \cite{Bial-BirRosen61}
which states that an injective polynomial mapping $P:\R^{n}\to\R^{n}$
must also be surjective (in the complex case, the same result was
proved a few years later and is well-known as the Ax-Grothendieck
theorm). As a consequence, one easily verifies that there exist no
injective polynomials from $\R^{n}$ into $\R^{m}$ with $m<n$. Indeed,
suppose that $P:\R^{n}\to\R^{m}$ is an injective polynomial mapping,
with $m<n$. Without loss of generality $m=n-1$. Set $x_{n}=0$.
The map $P(x_{1},x_{2},\dots,x_{n-1},0):\R^{n-1}\to\R^{n-1}$ is an
injective polynomial mapping, and hence it is surjective, a contradiction
to the fact that $P(x_{1},\dots,x_{n})$ is injective.
\end{proof}

\subsection{Adding an $\left(n+1\right)^{{\rm th}}$ direction\label{sec: n+1-directions}}

In this section we consider injections that map lines in given $n+1$
generic directions onto lines. As shown in Example \ref{exm:R3},
we cannot deduce without additional assumptions that such mappings
are affine-additive for $n\ge3$. However, using the extra direction
in which lines are mapped onto lines, we are able to describe further
restrictions on the possible polynomial form of these mappings, as
given in Theorem \ref{thm:Poly_n+1}.

For the proof of the Theorem \ref{thm:Poly_n+1}, we will need the
following lemma concerning bijections of the real line $\R$ (which
is valid over a general field).
\begin{lem}
\label{lem:Scalar-bijection-additive} Let $f:\R\to\R$ be a bijective
function with $f(0)=0$ and $f(1)=1$. Assume there is a function
$G:\R\to\R$ so that 
\[
\frac{f(a+b)-f(b)}{f(a)}=G(b)
\]
for every $a\neq0$ and every $b\in\R$. Then, $f$ is additive.\end{lem}
\begin{proof}
First, we rewrite the equation as 
\[
\frac{f(a+b)-f(b)-f(a)}{f(a)}=H(b),
\]
that is, 
\[
f(a+b)=f(a)+f(b)+f(a)H(b).
\]
From symmetry we get that (for $b\neq0$) 
\[
f(a+b)=f(a)+f(b)+f(b)H(a),
\]
and thus 
\[
f(a)H(b)=f(b)H(a)
\]
for $a,b\neq0$. This means that 
\[
\frac{H(a)}{f(a)}=\frac{H(b)}{f(b)}
\]
for $a,b\neq0$ which means this is a constant function, say $\alpha$,
so that $H(a)=\alpha f(a)$ for all $a\neq0$. Since $f(1)=1$, $\alpha=H(1)$.
We want to show that $H(1)=0$. Indeed, if $H(1)\neq0$ then for $b=f^{-1}(\frac{-1}{H(1)})$
we get $H(b)=-1$, which means that $f(a+b)=f(b)$, a contradiction
to the injectivity of $f$. So, $H(1)=0$, hence $H\equiv0$ and so
$f$ is additive.
\end{proof}

\begin{proof}
[Proof of Theorem \ref{thm:Poly_n+1}]By Fact \ref{fact:linear},
we may assume without loss of generality that $v_{i}=e_{i}$ for all
$i\in\left\{ 1,\dots,n\right\} $ and $v_{n+1}=v=\sum_{i=1}^{n}e_{i}$.
 By Theorem \ref{thm:poly-form}, $F$ has a representation 
\begin{equation}
F(x)=\sum_{\delta\in\{0,1\}^{n}}u_{\delta}\prod_{i=1}^{n}g_{i}^{\delta_{i}}(x_{i}).\label{eq:ProofOfMainReduced-rep}
\end{equation}
Our first goal is to show that $g_{i}$ are additive. To this end,
consider $A$ which maps $e_{i}$ to itself for $i=1,\ldots,n-1$
and $e_{n}$ to $v$. Then also $F\circ A$ has such a representation,
so that 
\[
F(x)=\sum_{\delta\in\{0,1\}^{n}}v_{\delta}\prod_{i=1}^{n}f_{i}^{\delta_{i}}((A^{-1}x)_{i})
\]
and as $(A^{-1}x)_{i}=x_{i}-x_{n}$ for $i=1,\dots,n-1$ and $(A^{-1}x)_{n}=x_{n}$
we get that 
\begin{equation}
\sum_{\delta\in\{0,1\}^{n}}u_{\delta}\prod_{i=1}^{n}g_{i}^{\delta_{i}}(x_{i})=\sum_{\delta\in\{0,1\}^{n}}v_{\delta}f_{n}^{\delta_{n}}(x_{n})\prod_{i=1}^{n-1}f_{i}^{\delta_{i}}(x_{i}-x_{n}).\label{eq:Additive-to-Identical}
\end{equation}
Plugging in all variables equal to $0$, we see that $v_{0}=u_{0}$,
and so we may assume that they are both equal to $0$. Plugging in
all variables but one equal to $0$ we see that $g_{j}=f_{j}$ for
$j=1,\ldots,n-1$, and that $u_{\delta^{j}}=v_{\delta^{j}}$ (recall
that $\delta^{j}$ is defined so that $\delta_{j}^{j}=1$ and $\delta_{i}^{j}=0$
for $i\neq j$). We also get that 
\[
u_{\delta^{n}}g_{n}(x_{n})=\sum_{\delta\in\{0,1\}^{n}}v_{\delta}f_{n}^{\delta_{n}}(x_{n})\prod_{i=1}^{n-1}g_{i}^{\delta_{i}}(-x_{n}).
\]
Next plug in only $x_{1},x_{n}\neq0$, using $f_{1}=g_{1}$ and the
notation $\delta^{1,n}=\delta^{1}+\delta^{n}$, to get 
\begin{equation}
v_{\delta^{1}}g_{1}(x_{1})+u_{\delta^{n}}g_{n}(x_{n})+u_{\delta^{1,n}}g_{1}(x_{1})g_{n}(x_{n})=\sum_{\delta\in\{0,1\}^{n}}v_{\delta}g_{1}(x_{1}-x_{n})^{\delta_{1}}f_{n}^{\delta_{n}}(x_{n})\prod_{i=2}^{n-1}g_{i}^{\delta_{i}}(-x_{n}).\label{eq:non-vanishing-expression}
\end{equation}
Using the previous equation we get, 
\[
v_{\delta^{1}}g_{1}(x_{1})+u_{\delta^{1,n}}g_{1}(x_{1})g_{n}(x_{n})=\sum_{\delta\in\{0,1\}^{n}}v_{\delta}[g_{1}(x_{1}-x_{n})^{\delta_{1}}-g_{1}(-x_{n})^{\delta_{1}}]f_{n}^{\delta_{n}}(x_{n})\prod_{i=2}^{n-1}g_{i}^{\delta_{i}}(-x_{n})
\]
which can be further reduced to 
\[
g_{1}(x_{1})[v_{\delta^{1}}+u_{\delta^{1,n}}g_{n}(x_{n})]=[g_{1}(x_{1}-x_{n})-g_{1}(-x_{n})]\sum_{\substack{\delta\in\{0,1\}^{n-1},\\
\delta_{1}=1
}
}v_{\delta}f_{n}^{\delta_{n}}(x_{n})\prod_{i=2}^{n-1}g_{i}^{\delta_{i}}(-x_{n})
\]
For non-zero $x_{1}$, divide by $g_{1}(x_{1})$ to get that 
\[
[v_{\delta^{1}}+u_{\delta^{1,n}}g_{n}(x_{n})]=\frac{[g_{1}(x_{1}-x_{n})-g_{1}(-x_{n})]}{g_{1}(x_{1})}\sum_{\substack{\delta\in\{0,1\}^{n-1},\\
\delta_{1}=1
}
}v_{\delta}f_{n}^{\delta_{n}}(x_{n})\prod_{i=2}^{n-1}g_{i}^{\delta_{i}}(-x_{n}).
\]
It is important to note here, that $v_{\delta^{1}}+u_{\delta^{1,n}}g_{n}(x_{n})\neq0$
for every $x_{n}$ since otherwise $F\left(x_{1}e_{1}+x_{n}e_{n}\right)$
would attain the same value independently of $x_{1}$ (as can be seen
from the left hand side of (\ref{eq:non-vanishing-expression})).
We thus see, as the left hand side does not depend on $x_{1}$ and
is non-zero, that also the right hand side does not depend on $x_{1}$,
and thus
\[
\frac{g_{1}(a+b)-g_{1}(b)}{g_{1}(a)}=G(b),
\]
for every $a\neq0$ and $b\in\R$ (for some function $G:\R\to\R$).
By Lemma \ref{lem:Scalar-bijection-additive}, $g_{1}$ is additive.
Similarly, $g_{j}$ is seen to be additive for all $j=2,\ldots,n-1$.
In fact, from the symmetry of the assumptions, it follows that also
$g_{n}$ must be additive. This amounts to considering a different
linear mapping $A$. 

Our next goal is to show that $u_{\delta}=0$ for all $|\delta|\ge\frac{n+2}{2}$.
To this end, let us rewrite Eq. \eqref{eq:Additive-to-Identical}
over the field $\Q$. Since the functions $g_{i},f_{i}$ are additive
and $f_{i}(1)=g_{i}(1)=1$, it follows that over $\Q$ they are the
identity functions. Hence 
\begin{equation}
\sum_{\delta\in\{0,1\}^{n}}u_{\delta}\prod_{i=1}^{n}x_{i}^{\delta_{i}}=\sum_{\delta\in\{0,1\}^{n}}v_{\delta}\,x_{n}^{\delta_{n}}\prod_{i=1}^{n-1}\left(x_{i}-x_{n}\right)^{\delta_{i}}\label{eq:Q-Id-n+1}
\end{equation}
for all $x=(x_{1},\dots,x_{n})\in\Q^{n}$.

Set $2\le k\le n$. Set $x_{k-1}=x_{k}=\cdots=x_{n}$. Our equation
takes the form 
\[
\sum_{\delta\in\{0,1\}^{n}}u_{\delta}\,x_{1}^{\delta_{1}}\cdots\,x_{k-2}^{\delta_{k-2}}\cdot x_{n}^{\sum_{k-1}^{n}\delta_{i}}=\sum_{\substack{\delta\in\{0,1\}^{n},\\
\delta_{k-1}=\cdots=\delta_{n-1}=0
}
}v_{\delta}\,x_{n}^{\delta_{n}}\prod_{i=1}^{k-2}\left(x_{i}-x_{n}\right)^{\delta_{i}}
\]
The right hand side is a polynomial of degree not greater than $k-1$.
Thus, comparing the coefficients of $x_{1}\,x_{2}\cdots x_{k-2}\,x_{n}^{2}$
we get 
\[
\sum_{\substack{|\delta|=k,\\
\delta_{1}=\delta_{2}=\cdots=\delta_{k-2}=1
}
}u_{\delta}=0.
\]
By the symmetries of $F$, we conclude that for any fixed $k-2$ coordinates
$i_{1},i_{2},\dots,i_{k-2}$, we have that 
\begin{equation}
A^{k}(i_{1},\dots,i_{k-2}):=\sum_{\substack{|\delta|=k,\\
\delta_{i_{1}}=\dots=\delta_{i_{k-2}}=1
}
}u_{\delta}=0.\label{eq:u_Zero_sums}
\end{equation}
For a given $2\le k\le n$, fix $l\le k-2$ coordinates (or none)
$i_{1},\dots,i_{l}$. Define 
\[
A^{k}(i_{1},\dots,i_{l}):=\sum_{\substack{|\delta|=k,\\
\delta_{i_{1}}=\cdots=\delta_{i_{l}}=1
}
}u_{\delta},\;\;\;A^{k}:=\sum_{|\delta|=k}u_{\delta}.
\]
There are $\binom{n-l}{k-2-l}$ choices for $(k-2)$ distinct indices
containing $i_{1},\dots,i_{l}$. All possible choices for the complement
$k-2-l$ indices are enumerated by $\{i_{l+1}^{p},\dots,i_{k-2}^{p}\}_{p=1,\dots,{\binom{n-l}{k-2-l}}}$.
Thus, \eqref{eq:u_Zero_sums} implies that 
\begin{equation}
A^{k}(i_{1},\dots,i_{l})=\frac{1}{\binom{k-l}{k-2-l}}\sum_{p=1}^{\binom{n-l}{k-2-l}}A^{k}(i_{1},\dots,i_{l},i_{l+1}^{p},\dots,i_{k-2}^{p})=0.\label{eq:u_Zero_sums2-1}
\end{equation}
Therefore, by the inclusion-exclusion principle, \eqref{eq:u_Zero_sums2-1}
implies that 
\begin{align*}
u_{\delta^{i_{1},\dots,i_{k}}}=A^{k}-\sum_{j=k+1}^{n}A^{k}(i_{j})+\sum_{\substack{j,l=k+1\\
j<l
}
}^{n}A^{k}(i_{j},i_{l})+\cdots+\left(-1\right)^{n-k}A^{k}(i_{k+1},\dots,i_{n})=0
\end{align*}
for every permutation $\{i_{1},\dots,i_{n}\}$ of $\{1,\dots,n\}$.
Concluding the above, we have that $u_{\delta}=0$ for every $\delta\in\{0,1\}^{n}$
with $|\delta|\ge\frac{n+2}{2}$.

Note that we have used the fact that $n-k\le k-2$ since $A^{k}\left(i_{1},\dots,i_{l}\right)$
is defined only for $l\le k-2$. Indeed, one cannot expect the conclusion
to hold without this assumption since the number of elements in the
set $\{u_{\delta}\;:\;|\delta|=k\}$ is $\binom{n}{k}$ whereas the
number of equations for this set, given in \eqref{eq:u_Zero_sums}
is $\binom{n}{k-2}$. Clearly, as long as $\binom{n}{k-2}<\binom{n}{k}$,
one may always find non-trivial solutions for these equations. However,
it is easily checked that $\binom{n}{k-2}\ge\binom{n}{k}$ if and
only if $k\ge\frac{n+2}{2}$, and in this case, we have that $n-k\le k-2$.
\end{proof}

\begin{example}
\label{rem:n+1:Sharpness} As mentioned in the introduction, Theorem
\ref{thm:Poly_n+1} is sharp in the sense that the degree of an injective
polynomial collineation in $n+1$ generic directions in $\R^{n}$
may be as high as $n/2$. Indeed, let us construct such a polynomial.
For simplicity, let us construct a polynomial $F:\R^{2n}\to\R^{2n}$
for some even dimension $2n$. Consider the following map: 
\[
F\left(x_{1},\dots,x_{2n}\right)=\left(x_{1},x_{2},\dots,x_{2n-1},x_{2n}+\sum_{\substack{\delta\in\left\{ 0,1\right\} ^{2n}\\
\left|\delta\right|=n
}
}\alpha_{n}\prod_{i=1}^{2n}x_{i}^{\delta_{i}}\right),
\]
where $\alpha_{\delta}\in\R$. Clearly, $F$ is an injective collineation
in directions $e_{1},\dots,e_{2n}$. 

Since $F\left(\left(c_{1},\dots,c_{2n}\right)+t\left(e_{1},\dots,e_{2n}\right)\right)$
is a polynomial in $t$, it is sufficient to choose $\alpha_{\delta}$
so that the coefficients of $t^{k}$ in the expansion of $F\left(\left(c_{1},\dots,c_{2n}\right)+t\left(e_{1},\dots,e_{2n}\right)\right)$
are all $0$, for every $2\le k\le n$ and any $\left(c_{1},\dots,c_{2n}\right)\in\R^{2n}$.
The resulting equation for a given $k$ is that the following sum
is equal to zero: the sum of all products of $n-k$ of the $c_{j}$'s,
with coefficients which are those $\alpha_{\delta}$ for which the
corresponding coefficients $\delta_{j}$ are one. That is, 
\[
c_{1}\cdots c_{n-k}\left(\sum_{\substack{\left|\delta\right|=n\\
\delta_{1}=\dots=\delta_{n-k}=1
}
}\alpha_{\delta}\right)+\dots+c_{k+n+1}\cdots c_{2n}\left(\sum_{\substack{\left|\delta\right|=n\\
\delta_{k+n+1}=\dots=\delta_{2n}=1
}
}\alpha_{\delta}\right)=0.
\]
Ensuring that each one of the above sums is equal to $0$ will conclude
our construction. In the last argument of the proof of Theorem \ref{thm:Poly_n+1},
we showed that this set of equations is equivalent to the set of equations
\[
\sum_{\substack{\left|\delta\right|=n\\
\delta_{i_{1}}=\dots=\delta_{i_{n-2}}=1
}
}\alpha_{\delta}=0,\,\,\,i_{1},\dots,i_{n-2}\in\left\{ 1,\dots,2n\right\} .
\]
As the number of variables $\alpha_{\delta}$ is ${2n \choose n}$,
which is greater than the number of equations ${2n \choose n-2}$,
it follows that there is a non-trivial choice of $\alpha_{\delta}$'s,
as claimed. 

In $\R^{4}$, the map 
\[
(x_{1},x_{2},x_{3},x_{4})\mapsto(x_{1},x_{2},x_{3},x_{4}-x_{2}x_{3}+x_{2}x_{4})
\]
is such a concrete construction.

For odd dimensions $n$, one may show that there is a similar construction
of such polynomials, but of degree $\lceil\frac{n-1}{2}\rceil$. It
is not clear whether there always exist such a polynomial of the maximal
degree (i.e., $\lceil\frac{n}{2}\rceil$) allowed in our theorem.
Proving such a statement, if true, should involve a non-trivial use
of the fact that the map is injective. Example \ref{exm:R3} confirms
this statement for $n=3$.  
\end{example}

\subsection{Low dimensional cases\label{sub:affine-additive}}

In this section we discuss the two and three dimensional cases in
which the map $F$ in Theorem \ref{thm:Poly_n+1} turns out to be
affine-additive. 

To ensure affine-additivity, one would need at least $n+{n \choose 2}$
directions in which lines are mapped onto lines. To see this, note
that just to have $u_{\delta}=0$ for $\left|\delta\right|=2$ in
the conclusion of Theorem \ref{thm:Poly_n+1}, requires additional
${n \choose 2}$ constraints on these coefficients, at the least.

For $n=2$, the conditions of Theorem \ref{thm:Poly_n+1} already
imply affine-additivity (see Theorem \ref{thm:FTAGPlane}). For $n=3$
we shall show that five $3-$independent directions suffice (see Theorem
\eqref{thm:NFTAG3D}).

\subsubsection{A fundamental theorem in the plane }

For $n=2$, the conclusion of Theorem \ref{thm:FTAGPlane} states
that $F$ is of the form 
\[
F\left(x_{1}v_{1}+x_{2}v_{2}\right)=u_{0}+f\left(x_{1}\right)u_{1}+g\left(x_{2}\right)u_{2}
\]
where $f,g$ are additive functions on $\R$, and $u_{i}\in\R^{2}$.
In other words, the following Theorem \ref{thm:Poly_n+1} is implied:
\begin{thm}
\label{thm:FTAGPlane}Let $n\ge2$, let $v_{1},v_{2},v_{3}\in\R^{2}$
be $2-$independent, and let $F:\R^{2}\to\R^{n}$ be injective. Assume
that $F$ maps each line in $\L\left(v_{1},v_{2},v_{3}\right)$ onto
a line. Then $F$ is affine-additive. \end{thm}
\begin{rem}
For the case of mappings from the plane to itself ($n=m=2$), a parallelism
condition is directly implied simply because any two lines in the
plane do not intersect if and only if they are parallel. Therefore,
in this case, Theorem \ref{thm:FTAGPlane} is an easy particular case
of Theorem \ref{thm:newFTAG.Parallelism-n+1dir}. 
\end{rem}

\subsubsection{A fundamental theorem in $\protect\R^{3}$}

In this section we deal with the three dimensional case, namely prove
Theorem \ref{thm:NFTAG3D}. In the proof we completely characterize
all forms of injective mappings taking lines onto lines in four $3-$independent
directions (see Remark \ref{rem:3D4directions}), and then show that
given one more direction in which lines are mapped onto lines, only
affine-additive forms are left (actually one line in this direction). 
\begin{proof}
[Proof of Theorem \ref{thm:NFTAG3D}]Without loss of generality, we
may assume that $F(0)=0$ and that $\{v_{1},v_{2},v_{3},v_{4}\}=\{e_{1},e_{2},e_{3},v\}$
where $v=e_{1}+e_{2}+e_{3}$. As $\{v_{1},v_{2},\dots,v_{5}\}$ is
$3$-independent, the direction $v_{5}$, in which $\{tv_{5}\}_{t\in\R}$
is mapped into a line, is of the form $u=e_{3}+ae_{1}+be_{2}$, with
$a,b\neq0,1$ and $a\neq b$. By Theorem \ref{thm:Poly_n+1} for this
particular case, $F$ is of the form 
\begin{align}
F(x_{1},x_{2},x_{3}) & =a_{1}f_{1}(x_{1})+a_{2}f_{2}(x_{2})+a_{3}f_{3}(x_{3})\label{eq:1stPolyFormR3}\\
\notag & +a_{4}f_{1}(x_{1})f_{2}(x_{2})+a_{5}f_{1}(x_{1})f_{3}(x_{3})+a_{6}f_{2}(x_{2})f_{3}(x_{3})
\end{align}
where $f_{1},f_{2},f_{3}:\R\to\R$ are additive bijections with $f_{i}(0)=0$
and $f_{i}(1)=1$, and $a_{1},a_{2},\dots,a_{6}\in\R^{3}$ with $a_{4}+a_{5}+a_{6}=0$.

Consider the mapping $\widetilde{F}(x_{1},x_{2},x_{3}):=F(f_{1}^{-1}(x_{1}),f_{2}^{-1}(x_{2}),f_{3}^{-1}(x_{3}))$.
Plugging the fact that $a_{4}+a_{5}+a_{6}=0$ into \eqref{eq:1stPolyFormR3}
yields that $\widetilde{F}$ is of the form 
\begin{equation}
\widetilde{F}(x_{1},x_{2},x_{3})=a_{1}x_{1}+a_{2}x_{2}+a_{3}x_{3}+a_{4}(x_{1}x_{2}-x_{2}x_{3})+a_{5}(x_{1}x_{3}-x_{2}x_{3}).\label{eq:TildePolyFormR3}
\end{equation}
Since $F$ is injective, it follows that $\widetilde{F}$ is also
injective. Denote $\widetilde{a}=f_{1}(a)$ and $\widetilde{b}=f_{2}(b)$
and note that since $F$ maps the line $\{t(a,b,1)\}_{t\in\R}$ into
a line and since $f_{1},f_{2},f_{3}$ are additive, it follows that
$\widetilde{F}$ maps all points in $\{N(\widetilde{a},\widetilde{b},1)\}_{N\in\N}$
into the same line. Also notice that $\widetilde{a},\widetilde{b}\neq0,1$,
which follows by the properties of $f_{1}$ and $f_{2}$. We continue
the proof by dividing into cases.

\noindent \textbf{Case 1}: $a_{4}=0$. Notice that in this case, $a_{1},a_{2},a_{3}$
are linearly independent, for otherwise we would have that $\tF(x_{1},x_{2},0)=\tF(0,0,x_{3})$
for some $x_{1},x_{2},x_{3}$, which would contradict the fact that
$\tF$ is injective. In particular, we may write $a_{5}=\alpha a_{1}+\beta a_{2}+\gamma a_{3}$
for some coefficients $\alpha,\beta,\gamma\in\R.$  Set $g\left(t\right)=f_{1}\left(t\right)-f_{2}\left(t\right)$.
Then for all $t\in\R$ we have that 
\begin{align*}
F\left(t,t,t\right) & =a_{1}\left[f_{1}\left(t\right)+\alpha f_{3}\left(t\right)g\left(t\right)\right]+a_{2}\left[f_{2}\left(t\right)+\beta f_{3}\left(t\right)g\left(t\right)\right]+a_{3}\left[f_{3}\left(t\right)+\gamma f_{3}\left(t\right)g\left(t\right)\right].
\end{align*}
Since $F$ maps the line $\{(t,t,t)\}_{t\in\R}$ into a line, which
passes through $F(0,0,0)=0$ and $F(1,1,1)=a_{1}+a_{2}+a_{3}$, it
follows that $F(t,t,t)\in\sp\{a_{1}+a_{2}+a_{3}\}$. Since $\left\{ a_{1},a_{2},a_{3}\right\} $
is a basis of $\R^{3}$, we may equate their coefficients in the above
formula for $F\left(t,t,t\right)$ and deduce that for all $t\in\R$,
$g\left(t\right)+\left(\alpha-\beta\right)f_{3}\left(t\right)g\left(t\right)=0,$
which by the injectivity of $f_{3}$ implies that $g\left(t\right)\equiv0$,
and hence $f_{1}\left(t\right)=f_{2}\left(t\right)=f_{3}\left(t\right)$
for all $t\in\R$. In particular, note that in this case $\tilde{a}\neq\tilde{b}$
as $a\neq b$ and $f_{1}=f_{2}$. We denote the identical functions
$f_{1},f_{2},f_{3}$ by $f$.

\noindent \textbf{Case 1.1}: $a_{5}=0$. In this case, the equation
$a_{4}+a_{5}+a_{6}=0$ implies that $a_{6}=0$ and so  \eqref{eq:1stPolyFormR3}
takes the form $F\left(x_{1},x_{2},x_{3}\right)=a_{1}f\left(x_{1}\right)+a_{2}f\left(x_{2}\right)+a_{3}f\left(x_{3}\right)$
which, in particular, implies that $F$ is additive since $f$ is
additive.

\noindent \textbf{Case 1.2}: $a_{5}\neq0$. Recall that $a_{5}=\alpha a_{1}+\beta a_{2}+\gamma a_{3}$,
due to which  \eqref{eq:TildePolyFormR3} takes the form 
\begin{align*}
\widetilde{F}(x_{1},x_{2},x_{3})=a_{1}[x_{1}+\alpha x_{3}(x_{1}-x_{2})]+a_{2}[x_{2}+\beta x_{3}(x_{1}-x_{2})]+a_{3}x_{3}[1+\gamma(x_{1}-x_{2})].
\end{align*}

\noindent Suppose $\gamma\neq0$. Note that $\widetilde{F}\left(0,\frac{1}{\gamma},1\right)\in\sp\left\{ a_{1},a_{2}\right\} $,
and therefore  
\[
\widetilde{F}(0,\frac{1}{\gamma},1)=\widetilde{F}(x_{1},x_{2},0)
\]
for some $x_{1},x_{2}\in\R$. This contradicts the fact that $\widetilde{F}$
is injective, and thus $\gamma=0$. Suppose $\alpha\neq\beta$. Then,
one may check that 
\[
\widetilde{F}\left(1,0,\frac{-1}{\alpha-\beta}\right)=\widetilde{F}\left(\frac{\alpha}{\alpha-\beta}+1,\frac{\beta}{\alpha-\beta},\frac{-1}{\alpha-\beta}\right)
\]
which contradicts the fact that $\tF$ is injective. Thus $\alpha=\beta$.
Note that $\alpha\neq0$ since we assumed the case $a_{5}\neq0$.

Since $\tF$ maps $\{N(\widetilde{a},\widetilde{b},1)\}_{N\in\N}$
into the line which passes through $\widetilde{F}(0,0,0)=0$ and $\widetilde{F}(\widetilde{a},\widetilde{b},1)$,
it follows that $\widetilde{F}(N\ta,N\tb,N)=\lambda(N)\widetilde{F}(\ta,\tb,1)$
for some $\lambda(N)$. However, 
\[
\widetilde{F}(N\ta,N\tb,N)=N\{a_{1}[\ta+\alpha N(\ta-\tb)]+a_{2}[\tb+\alpha N(\ta-\tb)]+a_{3}\}
\]
and so $\widetilde{F}(N\ta,N\tb,N)=\lambda(N)\widetilde{F}(\ta,\tb,1)$
implies $\ta=\tb$, a contradiction. \\

\noindent \textbf{Case 2}: $a_{4}\neq0$. Notice that in this case,
$\{a_{1},a_{2},a_{4}\}$ are linearly independent. Indeed, suppose
$a_{4}=\alpha a_{1}+\beta a_{2}$ with $\alpha,\beta$ not both $0$.
Suppose $\alpha\neq0$. Then 
\[
\tF(x_{1},x_{2},0)=a_{1}x_{1}(1+\alpha x_{2})+a_{2}x_{2}(1+\beta x_{1}).
\]
Note that $\widetilde{F}\left(1,-\frac{1}{\alpha},0\right)=\widetilde{F}\left(0,-\frac{\left(1+\beta\right)}{\alpha},0\right)$,
a contradiction to the fact that $\tF$ is injective. Similarly,
if $\beta\neq0$ one may find $x',x_{1}',x_{2}'$ so that $\tF(x_{1}',x_{2}',0)=\tF(x',0,0)$,
also a contradiction. Thus, we may write 
\[
a_{3}=\alpha_{3}a_{1}+\beta_{3}a_{2}+\gamma_{3}a_{4},\;\;a_{5}=\alpha_{5}a_{1}+\beta_{5}a_{2}+\gamma_{5}a_{4}
\]
for some $\alpha_{3},\beta_{3},\gamma_{3},\alpha_{5},\beta_{5},\gamma_{5}\in\R$.

\noindent \textbf{Case 2.1}: ${\sp\{a_{5}\}=\sp\{a_{4}\}}$. In this
case $\alpha_{5}=\beta_{5}=0$ and $\tF$ takes the form 
\[
\tF(x_{1},x_{2},x_{3})=a_{1}[x_{1}+\alpha_{3}x_{3}]+a_{2}[x_{2}+\beta_{3}x_{3}]+a_{4}[x_{1}x_{2}+\gamma_{5}x_{1}x_{3}-(1+\gamma_{5})x_{2}x_{3}+\gamma_{3}x_{3}].
\]
Let us check which restrictions are implied by the fact that $\widetilde{F}$
is injective. Suppose $\widetilde{F}(x_{1},x_{2},x_{3})=\widetilde{F}(y_{1},y_{2},y_{3})$.
We will use the fact that $\{a_{1},a_{2},a_{4}\}$ is a basis of $\R^{3}$,
to compare the values of $\tF$ in each of these coordinates separately.
Clearly, $x_{3}=y_{3}$ implies that $x_{1}=y_{1}$ and $x_{2}=y_{2}$.
Suppose $d:=x_{3}-y_{3}\neq0$.

The equations for the coefficients of $a_{1}$ and $a_{2}$ imply
that $x_{1}=y_{1}+\alpha_{3}d$ and $x_{2}=y_{2}+\beta_{3}d$, respectively.
Plugging these identities into the equation for the coefficients of
$a_{4}$, re-ordering the elements and eliminating a factor of $d\neq0$
yields the equation 
\begin{align*}
 & y_{1}(\beta_{3}-\gamma_{5})+y_{2}(\alpha_{3}+(1+\gamma_{5}))=\gamma_{3}-\alpha_{3}\beta_{3}d-\gamma_{5}\alpha_{3}x_{3}+(1+\gamma_{5})\beta_{3}x_{3}.
\end{align*}
The above equation implies that a necessary condition for the injectivity
of $F$ is that $\alpha_{3}=-(1+\gamma_{5})$ and $\beta_{3}=\gamma_{5}$.
Otherwise, we could find a solution for this equation. Thus, we are
left with the equation 
\begin{align*}
\gamma_{3}-\alpha_{3}\beta_{3}d-2\alpha_{3}\beta_{3}x_{3}=0.
\end{align*}
Clearly, since $x_{3}$ and $d\neq0$ are variables with no other
constraints, we must have that $\alpha_{3}\beta_{3}=0$ and $\gamma_{3}\neq0$
in order to have no solutions with $d\neq0$ for this equation. Summarizing
the above, in this case $\widetilde{F}$ is injective only if $\alpha_{3}=-(1+\gamma_{5})$,
$\beta_{3}=\gamma_{5}$, $\alpha_{3}\beta_{3}=0$ and $\gamma_{3}\neq0$.
In particular, $\alpha_{3}+\beta_{3}=-1$. Thus, either $\alpha_{3}=-1,\beta_{3}=0$
or $\beta_{3}=-1,\alpha_{3}=0$, and so $\widetilde{F}$ is either
of the form 
\begin{align*}
\widetilde{F}(x_{1},x_{2},x_{3}) & =a_{1}(x_{1}-x_{3})+a_{2}x_{2}+a_{4}(\gamma_{3}x_{3}+x_{2}(x_{1}-x_{3}))
\end{align*}
or of the form 
\begin{align*}
\widetilde{F}(x_{1},x_{2},x_{3}) & =a_{1}x_{1}+a_{2}(x_{2}-x_{3})+a_{4}(\gamma_{3}x_{3}+x_{1}(x_{2}-x_{3})).
\end{align*}
Note that the above two forms are the same up to a composition of
$\tF$, from both the left and the right, with the linear transformation
interchanging $a_{1}$ and $a_{2}$ (and fixing $a_{4}$). Under these
linear modifications, injectivity is preserved and the fifth direction
$(a,b,1)$ is interchanged with the direction $(b,a,1)$.

Suppose $\gamma_{3}=0$. Then, in the first form we would have that
$\tF(1,0,1)=\tF(0,0,0)$, and in the second form we would have that
$\tF(0,1,1)=\tF(0,0,0)$, which contradicts the fact that $\tF$ is
injective. Thus, $\gamma_{3}\neq0$. It is easy to verify that in
this case $\tF$ of the form above is injective, and so in order to
prove our theorem we need to invoke the fifth direction in which a
line is mapped into a line. As in Case 1, since $\tF$ maps all points
in $\{N(\ta,\tb,1)\}_{N\in\N}$ into the line that passes through
both $\tF(0,0,0)=0$ and $\tF(\ta,\tb,1)$, it follows that 
\[
\tF(N\ta,N\tb,N)=\lambda(N)\tF(\ta,\tb,1)
\]
for some $\lambda(N)$. However, in the first form we would have that
\[
\tF(N\ta,N\tb,N)=N[a_{1}(\ta-1)+a_{2}\tb+a_{4}(N\tb(\ta-1)+\gamma_{3})],
\]
and so $\tF(N\ta,N\tb,N)=\lambda(N)\tF(\ta,\tb,1)$ only if $\tb=0$
or $\ta=1$, which is impossible since $\widetilde{a},\widetilde{b}\neq0,1$.
By the linear connection of our two forms, we would get that in the
second form the points $\{N(\ta,\tb,1)\}_{N\in\N}$ are mapped into
a line only if $\ta=0$ or $\tb=1$, which is, again, impossible since
$\ta,\tb\neq0,1$, a contradiction.\\
%
\textbf{Case 2.2}: ${\sp\{a_{4}\}\neq\sp\{a_{5}\}}$. In this case,
we have that either $\alpha_{5}\neq0$ or $\beta_{5}\neq0$. By \eqref{eq:TildePolyFormR3},
$\widetilde{F}$ is of the form 
\begin{align*}
\tF(x_{1},x_{2},x_{3}) & =a_{1}(x_{1}+\alpha_{3}x_{3}+\alpha_{5}x_{3}(x_{1}-x_{2}))+a_{2}(x_{2}+\beta_{3}x_{3}+\beta_{5}x_{3}(x_{1}-x_{2}))\\
 & +a_{4}(x_{1}x_{2}+\gamma_{3}x_{3}+\gamma_{5}x_{1}x_{3}-(1+\gamma_{5})x_{2}x_{3}).
\end{align*}
We show that in this case, $F$ is not injective. By composing $\tF$
with the linear transformation interchanging $a_{1}$ and $a_{2}$
(and fixing $a_{4}$), we may assume without loss of generality that
$\beta_{5}\neq0$. We will show that there exist $x_{1},x_{2},x_{3}$
with $x_{3}\neq0$ such that $\tF(x_{1},x_{2},x_{3})\in\sp\{a_{1}\}$.
To find such points, denote $d=x_{1}-x_{2}$, and so we shall seek
a solution for the following system of equations: 
\begin{align}
 & x_{2}=-x_{3}\left(\beta_{5}d+\beta_{3}\right)\label{eq:2}\\
 & x_{3}\left(\gamma_{3}+\left(1+\gamma_{5}\right)d-x_{1}\right)=-x_{1}x_{2}\label{eq:3}\\
 & d=x_{1}-x_{2}\label{eq:4}
\end{align}
where the first two equations equate the coefficients of $a_{2}$
and $a_{4}$ to $0$, respectively. Plugging  \eqref{eq:2} into  \eqref{eq:3}
and dividing by $x_{3}\neq0$, we get 
\[
\gamma_{3}+(1+\gamma_{5})d-x_{1}=x_{1}(\beta_{5}d+\beta_{3}),
\]
and so 
\begin{equation}
x_{1}=\frac{\gamma_{3}+(1+\gamma_{5})d}{\beta_{5}d+\beta_{3}+1}.\label{eq:5}
\end{equation}
Plugging  \eqref{eq:2} and \eqref{eq:5} into  \eqref{eq:4} yields
\begin{equation}
\frac{\gamma_{3}(1+\gamma_{5})d}{\beta_{5}d+\beta_{3}+1}+x_{3}(\beta_{5}d+\beta_{3})=d.\label{eq:6}
\end{equation}
As $\beta_{5}\neq0$, we may choose $d\in\R$, say, large enough,
such that both  \eqref{eq:2} and  \eqref{eq:5} are well defined,
and such that \eqref{eq:6} holds for some $x_{3}\neq0$. Concluding
the above, we showed that there exist $x_{1},x_{2},x_{3}$ with $x_{3}\neq0$
such that $\tF(x_{1},x_{2},x_{3})=ca_{1}$ for some $c\in\R$. Thus,
$\tF(c,0,0)=\tF(x_{1},x_{2},x_{3}),$ a contradiction to the fact
that $\tF$ is injective. 
This completes the consideration of this case, and hence the proof
as well.\end{proof}
\begin{rem}
\label{rem:3D4directions}In this proof of Theorem \ref{thm:NFTAG3D}
we actually completely classify all possible forms of injective mappings
of $\R^{3}$ that map all lines in four direction in general position
onto lines. The proof shows that, up to obvious linear modifications
and compositions of the coordinates with bijective maps attaining
$0$ at $0$ and $1$ at $1$, such maps are either of the form 
\begin{align*}
F(x_{1},x_{2},x_{3}) & =(x_{1}+\alpha x_{3}(x_{1}-x_{2}),x_{2}+\alpha x_{3}(x_{1}-x_{2}),x_{3}),
\end{align*}
or of the form 
\begin{align*}
F(x_{1},x_{2},x_{3}) & =(x_{1}-x_{3},\,x_{2},\,\alpha x_{3}+x_{2}(x_{1}-x_{3})),
\end{align*}
where in both forms $\alpha\neq0$. We also point out that, in the
proof, only one specific line in the fifth direction, in which parallel
lines are mapped to lines, was needed for the proof, namely the line
trough the origin. 
\end{rem}

\subsection{An example for a sufficient set of directions in $\protect\R^{n}$}

In this section, we give an example for a finite set of directions
in $\R^{n}$ for which an injective collineation, in these directions,
must be affine-additive. Namely, consider the following set of $n+{n \choose 2}+1$
directions. 
\[
S=\left\{ e_{i}\,:\,i\in\left\{ 1,\dots,n\right\} \right\} \cup\left\{ e_{i}+e_{j}\,:\,i,j\in\left\{ 1,\dots,n\right\} \right\} \cup\left\{ e_{1}+\dots+e_{n}\right\} 
\]
We prove the following:
\begin{thm}
Let $F:\R^{n}\to\R^{n}$ be injective. Suppose that $F$ maps each
line in $\L\left(S\right)$ onto a line. Then $F$ is given by 
\[
F\left(x_{1},\dots,x_{n}\right)=\sum_{i=1}^{n}f\left(x_{i}\right)v_{i},
\]
for some additive bijection $f:\R\to\R$, and $v_{1},\dots,v_{n}\in\R^{n}$. \end{thm}
\begin{proof}
By Theorem \ref{thm:Poly_n+1}, $F$ is of the form 
\begin{equation}
F\left(x_{1},\dots,x_{n}\right)=\sum_{\delta\in\left\{ 0,1\right\} ^{n}}u_{\delta}\prod f_{i}\left(x_{i}\right)^{\delta_{i}}\label{eq:Binary_Poly}
\end{equation}
for some $u_{\delta}\in\R^{n}$ and additive bijections $f_{i}:\R\to\R$
with $f_{i}\left(1\right)=1$.

Next, we observe that $F$ maps parallel lines in $\L\left(e_{1},\dots,e_{n}\right)$
onto parallel lines. Indeed, we let $v=\alpha_{1}e_{1}++\alpha_{n}e_{n}\in\R^{n}$,
and show that $F\left(\R e_{1}\right)$ and $F\left(v+\R e_{1}\right)$
are parallel. Since $\R e_{1}$, $e_{2}+\R e_{1}$, $\R e_{2}$, $\R\left(e_{1}+e_{2}\right)$,
and $\frac{1}{2}e_{2}+\R\left(e_{1}+e_{2}\right)$ lie on one plane,
and are all mapped onto lines, under $F$, it follows that $F\left(\R e_{1}\right)$,
$F\left(e_{1}+e_{2}+\R e_{1}\right)$, $F\left(\R e_{2}\right)$,
$F\left(\R\left(e_{1}+e_{2}\right)\right)$, and $F\left(\frac{1}{2}e_{2}+\R\left(e_{1}+e_{2}\right)\right)$
also lie on one plane. In particular, it follows that $F\left(\R e_{1}\right)$
and $F\left(e_{1}+e_{2}+\R e_{1}\right)$ are parallel, and hence
$F\left(\R e_{1}\right)$ is parallel to $F\left(\alpha_{1}e_{1}+\alpha_{2}e_{2}+\R e_{1}\right)$.
Similarly, $\alpha_{1}e_{1}+\alpha_{2}e_{2}+\R e_{1}$, $e_{3}+\alpha_{1}e_{1}+\alpha_{2}e_{2}+\R e_{1}$,
$\R e_{3}$, $\alpha_{1}e_{1}+\alpha_{2}e_{2}+\R\left(e_{1}+e_{3}\right)$,
and $\frac{1}{2}e_{3}+\alpha_{1}e_{1}+\alpha_{2}e_{2}+\R\left(e_{1}+e_{3}\right)$
lie on one plane, which leads to the conclusion that $F\left(\alpha_{1}e_{1}+\alpha_{2}e_{2}+\alpha_{3}e_{3}+\R e_{1}\right)$
is parallel to $F\left(\alpha_{1}e_{1}+\alpha_{2}e_{2}+\R e_{1}\right)$,
and therefore parallel to $F\left(\R e_{1}\right)$. By applying the
above argument iteratively, we conclude that $F\left(v+\R e_{1}\right)$
and $F\left(\R e_{1}\right)$ are parallel. Similarly, $F\left(v+\R e_{i}\right)$
and $F\left(\R e_{i}\right)$ are parallel, for any $i\in\left\{ 1,\dots,n\right\} $.
Thus, we may apply Theorem \ref{thm:newFTAG.Parallelism}, which implies
that $F$ is of the form 
\[
F\left(x_{1},\dots,x_{n}\right)=u_{0}+g_{1}\left(x_{1}\right)u_{1}+\dots+g_{n}\left(x_{n}\right)u_{n}
\]
for some bijections $g_{i}:\R\to\R$, and some $u_{i}\in\R^{n}$.
By comparing the above form with the form given in \eqref{eq:Binary_Poly},
it follows that $g_{i}=f_{i}$ for all $i$, which means that $g_{i}$
are additive, and hence $F$ is affine-additive. 

Finally, we show that all the functions $f_{1},f_{2},\dots,f_{n}$
are identical. Since $f_{i}$ are all additive with $f_{i}\left(1\right)=1$,
we have on the one hand that for any rational number $q$, 
\[
F\left(q,q,\dots,q\right)=u_{0}+q\left(u_{1}+\dots+u_{n}\right).
\]
On the other hand, since $F$ maps the line $\R\left(e_{1}+\dots+e_{n}\right)$
onto a a line, it follows that for any $x\in\R^{n}$, $F\left(x,x,\dots,x\right)=u_{0}+f\left(x\right)u_{1}+\dots+f\left(x\right)u_{n}$
parallel to $u_{1}+\dots+u_{n}$. Therefore, we conclude that $f_{1}\left(x\right)=\dots=f_{n}\left(x\right)$,
as claimed. 
\end{proof}

\section{Fundamental theorems of projective geometry\label{sec:FTPG}}

\subsection{Projective point of view\label{sec:Proj}}

In his section we consider analogues of our results in the projective
space. It is natural to consider the projective space when discussing
maps which preserve lines.  In fact, the classical fundamental theorem,
as well as additional earlier results in the spirit of this note,
were originally formulated for the projective plane, as described
in Section \ref{sec:History}. In this section we discuss and prove
such analogous results in the projective setting.

Let us explain our projective framework. The projective space corresponding
to a linear space $E$ over $\R$ (or any other field), is denoted
by $P(E)$. Each point in $P(E)$ corresponds to a distinct one-dimensional
subspace of $E$. For $E=\R^{n}$, the projective space is often denoted
by $P(\R^{n})=\RP^{n-1}$. Since we would like to employ here the
main results of this note for $\R^{n}$, it will be useful for us
to use the following standard embedding of $\R^{n}$ into $\R P^{n}$:
\begin{equation}
\RP^{n}=\R^{n}\cup\RP^{n-1}.\label{eq:emRnPn}
\end{equation}

To make sense of this embedding, one can consider two copies of $\R^{n}$
in $\R^{n+1}$. One is the ``base\textquotedbl{}: $\R^{n}=\sp\{e_{1},...,e_{n}\}$,
and one ``affine'' copy, placed one unit above: $e_{n+1}+\R^{n}$.
Each line through the origin which lies in $\R^{n}$ corresponds to
a projective point in $\RP^{n-1}$, and each line which does not lie
in $\R^{n}$ intersects $e_{n+1}+\R^{n}$ at a distinct point corresponding
to a point in the affine copy of $\R^{n}$. In this way, $\RP^{n}$
is obtained as a compactification of the affine copy of $\R^{n}$,
by adding to it all directions at infinity represented by $\RP^{n-1}$.

Recall that in the case of $\R^{n}$, we usually assumed that parallel
lines (in certain directions) are mapped to not-necessarily-parallel
lines, which was a major difficulty. However, in the projective setting
this difficulty does not exist since the natural projective analogue
of a family of parallel lines in $\R^{n}$ is a family of projective
lines which intersect at a common projective point. To see this, take
a family of projective lines in $\RP^{n}$, which correspond to a
family of parallel lines in the affine copy of $\R^{n}$. As these
lines are all parallel, they must intersect at a single projective
point at infinity. Therefore, our geometric assumption will naturally
be that all projective lines passing through finitely many given projective
points are mapped onto projective lines.

The analogue in $\R P^{n}$ of affine (invertible) transformations
in $\R^{n}$ will be projective-linear transformations, denoted by
${\rm PGL_{n+1}}\left(\R\right)$, namely mappings of $\R P^{n}$
which are induced by (invertible) linear transformations of $\R^{n+1}$.
In particular, such transformations map projective lines onto projective
lines.

\subsection{Basic facts and preliminary results}

Recall that $n+2$ projective points $\bar{a}_{1},...,\bar{a}_{n+2}\in\RP^{n}$
are said to be in general position if any of their lifts $a_{1},\dots,a_{n+2}\in\R^{n+1}$
are in general position (see Section \ref{sec:Notations}). We will
have use of the following theorem (see e.g., \cite{Pras01}) and basic
facts regarding projective-linear transformations in $\R P^{n}$.
\begin{thm}
\label{thm:PrLiPo} Let $\bar{a}_{1},...,\bar{a}_{n+2}$ and $\bar{b}_{1},...,\bar{b}_{n+2}$
be two sets of points in general position in $\RP^{n}$. Then, there
exists a unique projective-linear transformation $f:\RP^{n}\to\RP^{n}$
such that $f(\bar{a}_{i})=b_{i}$ for $i=1,...,n+2$.\end{thm}
\begin{fact}
\label{fact:PrEi} Let $F:\RP^{n}\to\RP^{n}$ map $n+1$ points $\bar{p}_{1},...,\bar{p}_{n+1}\in\RP^{n}$
in general position to $n+1$ points $\bar{q}_{1},...,\bar{q}_{n+1}\in\RP^{n}$
in general position. Then we may assume without loss of generality
that $\bar{p}_{i}=\bar{e}_{i}$ and $\bar{q}_{i}=\bar{e}_{i}$ for
$i=1,...,n+1$ by composing $F$ with projective-linear transformations
from the left and from the right: $A\circ F\circ B$ (which exist
by Theorem \ref{thm:PrLiPo}).
\end{fact}

\begin{fact}
\label{fact:PrScale} Let $F:\RP^{n}\to\RP^{n}$. Assume that the
affine copy of $\R^{n}$ (from the representation given in \eqref{eq:emRnPn})
is invariant under $F$, and denote its restriction this affine copy
by $F':\R^{n}\to\R^{n}$. Let $D\in GL_{n}$ be a diagonal matrix
in $\R^{n}$ and let $b\in\R^{n}$. Then there exists a projective-linear
transformation $\bar{A}\in\GPL$ such that the restriction $(\bar{A}F)'$
of $\bar{A}F$ to the affine copy of $\R^{n}$ satisfies $(\bar{A}F)'=DF'+b$.
Namely, the inducing linear transformation $A\in GL_{n+1}$ is 
\[
A=\left(\begin{array}{ccc|c}
 &  & \\
 & D &  & b\\
 &  & \\
\hline 0 & \cdots & 0 & 1
\end{array}\right)
\]

\end{fact}
To prove Theorems \ref{thm:Pr+2} and \ref{thm:Pr+1} we will need
some preliminary results.
\begin{prop}
\label{prop:PrParll} Let $n\ge2$ and let an injective mapping $F:\RP^{n}\to\RP^{n}$
be given. Let $\bar{p}_{1},\ldots,\bar{p}_{m}\in\RP^{n}$ be $m\leq n+1$
generic points, and assume that $F$ maps any projective line passing
through one of the points $\bar{p}_{i}$ onto a projective line. Denote
$F(\bar{p}_{i})=\bar{q}_{i}$. Then $\{\bar{q}_{i}\}_{i=1}^{m}$ are
generic and 
\begin{equation}
F(\Psp\{\bar{p}_{i}\}_{i=1}^{m})=\Psp\{\bar{q}_{i}\}_{i=1}^{m}.\label{eq:PrParll}
\end{equation}
\end{prop}
\begin{proof}
The proof is similar to the proof of Lemma \ref{lem:parallelism},
and goes by induction on $m$. The case $m=1$ is trivial. Assume
the lemma is true for $(m-1)$ and let generic $\bar{p}_{1},\ldots,\bar{p}_{m}\in\RP^{n}$
be given. Since by the induction hypothesis 
\[
F(\Psp\{\bar{p}_{i}\}_{i=1}^{m-1})=\Psp\{\bar{q}_{i}\}_{i=1}^{m-1},
\]
the injectivity assumption implies that $F(\bar{p}_{m})\not\in\Psp\{\bar{q}_{i}\}_{i=1}^{m-1}$,
and so $\{\bar{q_{i}}\}_{i=1}^{m}$ are generic.

Next, we prove the equality in (\ref{eq:PrParll}) by showing double
inclusion. For the inclusion of the L.H.S. in the R.H.S. in equation
(\ref{eq:PrParll}) let $\bar{x}\in\Psp\{p_{i}\}_{i=1}^{m}\setminus\left\{ \bar{p_{i}}\right\} _{i=1}^{m}$
and let $l$ be the projective line connecting $\bar{x}$ and $\bar{p}_{m}$.
Since $\left\{ \bar{p}_{i}\right\} _{i=1}^{m}$ are generic, there
exists a point $\bar{y}\in\Psp\{\bar{p}_{i}\}_{i=1}^{m-1}$ so that
$\bar{y}\in l$. By the induction hypothesis, $F(\bar{y})\in\Psp\{\bar{q}_{i}\}_{i=1}^{m-1}$.
Since $F(\bar{p}_{m})\in\Psp\{\bar{q}_{i}\}_{i=1}^{m}$, and $F$
maps every projective lines through $p_{m}$ onto a projective line,
it follows that $F(l)\subset\Psp\{\bar{q}_{i}\}_{i=1}^{m}$. In particular,
$F(\bar{x})\in\Psp\{\bar{q}_{i}\}_{i=1}^{m}$.

For the second inclusion, let $\bar{y}'\in\Psp\{\bar{q}_{i}\}_{i=1}^{m}\setminus\left\{ \bar{q}_{i}\right\} _{i=1}^{m}$
and let $l'$ be the projective line connecting $\bar{y}'$ and $\bar{q}_{m}$.
Then there exists a projective point $\bar{z}'\in\Psp\{\bar{q}_{i}\}_{i=1}^{m-1}$
so that $\bar{z}'\in l'$. By the induction hypothesis, there exists
a projective point $\bar{z}\in\Psp\{\bar{p}_{i}\}_{i=1}^{m-1}$ such
that $F(\bar{z})=\bar{z}'$. Thus, the projective line $l$ passing
through $\bar{p}_{m}$ and $\bar{z}$ is mapped onto a projective
line passing through $\bar{q_{m}}$ and $\bar{z}'$, namely $l'$.
In particular, there exists a projective point $\bar{y}\in\Psp\{\bar{p}_{i}\}_{i=1}^{m}$
for which $F(\bar{y})=\bar{y}'$.
\end{proof}
The following auxiliary lemmas, which are concerned with mappings
of $\R$ and $\R^{n}$, will allow us to bridge between results of
previous sections to the projective setting.
\begin{lem}
\label{lem:DStrMult} Let $n\geq2$, Let $F:\R^{n}\to\R^{n}$ be a
bijection satisfying that for every $x=(x_{1},x_{2},...,x_{n})^{T}\in\R^{n}$,
\[
F(x)=\sum_{i=1}^{n}f_{i}(x_{i})e_{i}
\]
where $f_{i}:\R\to\R$ are bijections with $f_{i}(0)=0$ and $f_{i}(1)=1$.
Assume that $F$ maps any line through $0$ into a line. Then, $f_{1}=f_{2}=\cdots=f_{n}$
and $f_{1}$ is multiplicative.\end{lem}
\begin{proof}
First, since $F$ maps the line passing through $e=e_{1}+e_{2}+\cdots+e_{n}$
and through the origin into a line, and since $F$ fixes the origin
and the point $e$, it follows that $F$ maps the line $\sp\{e\}$
into itself. Therefore, $f_{1}=f_{2}=\cdots=f_{n}$ and we denote
this function by $f$. Next, since lines through $0$ are mapped into
lines through $0$, it follows that for every $\alpha=(\alpha_{1},\alpha_{2},...,\alpha_{n})^{T}\in\R^{n}$
and every $t\in\R$ we have 
\[
\sum_{i=1}^{n}f(t\alpha_{i})e_{i}=F(t\alpha)=s(t)F(\alpha)=s(t)\sum_{i=1}^{n}f(\alpha_{i})e_{i}
\]
for some $s:\R\to\R$. Note that $s$ is independent of the choice
of $\alpha$. Indeed, choosing $\beta=(\beta_{1},\beta_{2},...,\beta_{n})^{T}$
with $\beta_{1}=\alpha_{1}\neq0$ we have 
\[
F(t\beta)=\sum_{i=1}^{n}f(t\beta_{i})e_{i}=g(t)\sum_{i=1}^{n}f(\beta_{i})e_{i}
\]
for some $g:\R\to\R$ which, as $s$, satisfies $g(t)f(\alpha_{1})=f(t\alpha_{1})=s(t)f(\alpha_{1})$
and so $g(t)=s(t)$. The fact that we could choose, for example, any
$\beta_{2}$ means that $f(t\beta_{2})=s(t)f(\beta_{2})$ for any
$t,\beta_{2}\in\R$. Plugging $\beta_{2}=1$ we get $s(t)=f(t)$ and
so $f$ is multiplicative.\end{proof}
\begin{lem}
\label{lem:2MultId} Let $f:\R\to\R$ be a multiplicative injection.
Assume $g(x)=f(x+1)-1$ is also multiplicative. Then $f(x)=x$ for
every $x\in\R$.\end{lem}
\begin{proof}
It is easy to check that the injectivity and multiplicativity of $f$
and $g$ imply that $f(0)=g(0)=0$, $f(\pm1)=g(\pm1)=\pm1$ and that
$f(-a)=-f(a)$ and $g(-a)=-g(a)$ for every $a\in\R$. 
Next, using the relation $g(x)+1=f(x+1)$ and the multiplicaitivity
of $f$ and $g$ we get that 
\[
g(ab+a+b)+1=f(ab+a+b+1)=f(a+1)f(b+1)=g(a)+g(b)+g(ab)+1
\]
for any $a,b\in\R$. Plugging in $y=b$ and $x=(b+1)a$ we get 
\[
g(x+y)=g(\frac{x}{y+1})+g(y)+g(\frac{x}{y+1}y)
\]
for every $x$ and every $y\neq-1$. Thus for any $x\neq0$ and $y\neq-1$
we have 
\[
\frac{g(x+y)-g(y)}{g(x)}=g(\frac{1}{1+y})+g(\frac{y}{1+y}).
\]
We would like to use Lemma \ref{lem:Scalar-bijection-additive} to
conclude that $g$ is additive, but we still have to deal with the
case $y=-1$. Repeating the idea of the proof of Lemma \ref{lem:Scalar-bijection-additive},
we have that for every $x,y\neq0,-1$  
\[
\frac{g(x+y)-g(y)-g\left(x\right)}{g(x)}=g(\frac{1}{1+y})+g(\frac{y}{1+y})-1=:H\left(y\right)
\]
and so $g\left(x+y\right)-g\left(y\right)-g\left(x\right)=g\left(x\right)H\left(y\right)$.
By interchanging the roles of $x$ and $y$, we obtain the equation
$g\left(x+y\right)-g\left(y\right)-g\left(x\right)=g\left(y\right)H\left(x\right)$,
and so for all $x,y\neq0,-1$ we have that 
\[
\frac{H\left(y\right)}{g\left(y\right)}=\frac{H\left(a\right)}{g\left(a\right)}.
\]
Thus, $H\left(x\right)=\alpha g\left(x\right)$ for some constant
$\alpha\in\R$ and for all $x\neq0,-1$. By plugging $x=1$, and using
the aforementioned properties of $g$ we conclude that
\[
\alpha=H\left(1\right)=g\left(1/2\right)+g\left(1/2\right)-1=\frac{1}{2}g\left(1\right)+\frac{1}{2}g\left(1\right)-1=0,
\]
and so $g\left(x+y\right)=g\left(x\right)+g\left(y\right)$ for all
$x,y\neq0,-1$. Since $g\left(0\right)=0$, we obviously have that
$g\left(a+0\right)=g\left(a\right)+g\left(0\right)$ for all $a\in\R$.
Moreover, for all $a\neq1$, 
\[
g\left(-1+a\right)=g\left(-\left(1-a\right)\right)=-g\left(1-a\right)=-g\left(1\right)-g\left(a\right)=g\left(-1\right)+g\left(a\right),
\]
and since the equality $g\left(1-1\right)=g\left(1\right)+g\left(-1\right)$
holds as well, we conclude that $g$ is additive. Since $g$ is also
multiplicative, it is the identity (this is well known and easy to
prove, for example one can show that the multiplicativity implies
monotonicity and together with the additivity one gets linearity)
and so $f(x)=g(x+1)-1=x$.
\end{proof}
The next lemma is immediately implied by Lemma \ref{lem:2MultId}.
\begin{lem}
\label{lem:f(2)Id} Let $f:\R\to\R$ be a multiplicative injection.
Assume that also the function ${\displaystyle g(x)=\frac{f(x+1)-1}{f(2)-1}}$
is multiplicative. Then $f(x)=x$ for every $x\in\R$.\end{lem}
\begin{proof}
As stated in the proof of Lemma \ref{lem:2MultId}, the injectivity
and multiplicativity of $f$ and $g$ imply that $f(0)=0$ and $g(-1)=-1$.
Plugging $x=-1$ into the formula of $g$ implies that $f(2)=2$ and
so $g(x)=f(x+1)-1$. By Lemma \ref{lem:2MultId}, $f(x)=x$ for every
$x\in\R$.\end{proof}
\begin{lem}
\label{lem:Add1Str} Let $n\geq2$, Let $F:\R^{n}\to\R^{n}$ be an
additive bijection. Assume $F$ maps any line passing through a given
point $x_{0}\in\R^{n}$ into a line. Then, $F$ is affine. \end{lem}
\begin{proof}
Let $l\subset\R^{n}$ be any line and choose any $x\in l$. By the
additivity of $F$, $F(l)=F(l-x+x_{0})+F(x-x_{0})$ which, by our
assumption, is contained in a line. Thus, by the classical fundamental
lemma of affine geometry (namely, Theorem \ref{thm:ClassFTAG} below),
$F$ is affine.\end{proof}
\begin{prop}
\label{lem:Diag2Str} Let $n\geq2$, Let $F:\R^{n}\to\R^{n}$ be a
bijection satisfying that for every $x=(x_{1},x_{2},...,x_{n})^{T}\in\R^{n}$,
\[
F(x)=\sum_{i=1}^{n}f_{i}(x_{i})e_{i}
\]
where $f_{i}:\R\to\R$ are bijections with $f_{i}(0)=0$ and $f_{i}(1)=1$.
Let $x_{0}\in\R^{n}\setminus\{0\}$ be a vector of the form $x_{0}=\sum_{i=1}^{n}\alpha_{i}e_{i}$
where for each $i$, $\alpha_{i}=1$ or $0$. Assume that $F$ maps
any line through the origin and any line through $x_{0}$ into a line.
Then, $F(x)=x$ for any $x$.\end{prop}
\begin{proof}
By Lemma \ref{lem:DStrMult}, all of the $f_{i}$'s are identical,
so we denote them by $f$, and $f$ is multiplicative. Next, we define
a function $G:\R^{n}\to\R^{n}$, as follows. 
\[
G\left(\sum_{i=1}^{n}x_{i}e_{i}\right)=\sum_{i=1}^{n}g_{i}(x_{i})e_{i}:=\sum_{i=1}^{n}\frac{f(x_{i}+\alpha_{i})-\alpha_{i}}{f(1+\alpha_{i})-\alpha_{i}}e_{i}.
\]
It is easy to check that for each $i$, $g_{i}(0)=0$ and $g_{i}(1)=1$.
Moreover, one can check that $G$ maps lines through $0$ into lines,
since $F$ maps lines through $x_{0}$ into lines.  By Lemma \ref{lem:DStrMult},
all the $g_{i}$'s are identical, and multiplicative. Since there
exists at least one index $i$ for which $\alpha_{i}=1$, it follows
that ${\displaystyle \frac{f(x+1)-1}{f(2)-1}}$ is multiplicative,
as well as $f$. By Lemma \ref{lem:f(2)Id}, $f(x)=x$.
\end{proof}

\subsection{Proofs of the projective main results}

In this section we prove Theorems \ref{thm:Pr+2} and \ref{thm:Pr+1}.
\begin{proof}
[Proof of Theorem \ref{thm:Pr+2}] By Proposition \ref{prop:PrParll},
$F(\bar{p}_{1}),...,F(\bar{p}_{n+1})$ are in general position. Hence,
by Fact \ref{fact:PrEi} we may assume without loss of generality
that $\bar{p}_{i}=\bar{e}_{i}$ for $i=1,...,n+1$ and that $F(\bar{e}_{i})=\bar{e}_{i}$
for $i=1,...,n+1$. We identify $\Psp\{\bar{e}_{1},...,\bar{e}_{n}\}$
with $\R P^{n-1}$ and recall the representation \eqref{eq:emRnPn}
where $\R P^{n}=\R P^{n-1}\cup\R^{n}$. Proposition \ref{prop:PrParll}also
implies that $F(\R P^{n-1})=\R P^{n-1}$, and so $F(\R^{n})=\R^{n}$.
By the theorem's assumption, $\bar{p}_{n+2}$ belongs to the affine
copy of $\R^{n}$, and since $\bar{p}_{n+2}\neq\bar{e}_{n+1}$, it
corresponds to a point $x_{0}\in\R^{n}\setminus\left\{ 0\right\} $.
By composing $F$ with a diagonal projective-linear transformation
from the right, we may assume without loss of generality that $x_{0}=e_{1}+\cdots+e_{k}$. 

Denote the restriction of $F$ to the affine copy of $\R^{n}$ by
$F':\R^{n}\to\R^{n}$, and observe that $F'$ satisfies the conditions
of Theorem \ref{thm:newFTAG.Parallelism} and so it is of diagonal
form: 
\[
F'(x_{1},....,x_{n})=\sum_{i=1}^{n}f_{i}(x_{i})e_{i},
\]
where $f_{i}:\R\to\R$ are bijections. Since $F(\bar{e}_{n+1})=\bar{e}_{n+1}$,
$F'(0)=0$ and $F'$ maps lines through the origin into lines. Moreover,
$F'$ maps any line through its corresponding point $x_{0}\in\R^{n}$
into a line. By Fact \ref{fact:PrScale} we may assume without loss
of generality that $F'(e_{i})=e_{i}$ for $i=1,...,n$ by composing
$F$ with a diagonal projective-linear transformation from the left.
Thus $f_{i}(0)=0$ and $f_{i}(1)=1$ for all $i$. Then, $F'$ satisfies
the conditions of Lemma \ref{lem:Diag2Str} which implies that $F'$
is the identity mapping. Next, we explain why $F|_{\R P^{n-1}}$ is
also the identity. Let $\bar{p}\in\R P^{n-1}$ and take a projective
line $\bar{l}$ including $\bar{p}$ and $\bar{p}_{n+2}$. Since the
restriction of $F$ to affine copy of $\R^{n}$ is the identity map,
any projective point in $\bar{l}$ different than $\bar{p}$ is mapped
to itself. Since, by assumption, $\bar{l}$ is mapped onto a line,
it follows that $\bar{p}$ must be mapped to itself as well. Thus,
up to compositions with projective-linear transformations, $F$ is
the identity map, which means that $F$ is a projective-linear mapping.
\end{proof}

\begin{proof}
[Proof of Theorem \ref{thm:Pr+1}] By Proposition \ref{prop:PrParll},
$F(\bar{p}_{1}),...,F(\bar{p}_{n+1})$ are generic in the projective
subspace $\Psp\{F(\bar{p}_{1}),...,F(\bar{p}_{n})\}$. Hence, by Fact
\ref{fact:PrEi} we may assume without loss of generality that $\bar{p}_{i}=\bar{e}_{i}$
and $F(\bar{e}_{i})=\bar{e}_{i}$ for $i=1,...,n+1$. We identify
$\Psp\{\bar{e}_{1},...,\bar{e}_{n}\}$ with $\R P^{n-1}$ and recall
the representation \eqref{eq:emRnPn} where $\R P^{n}=\R P^{n-1}\cup\R^{n}$.
Proposition \ref{prop:PrParll} also implies that $F(\R P^{n-1})=\R P^{n-1}$
and so $F(\R^{n})=\R^{n}$. 

Denote the restriction of $F$ to its affine copy of $\R^{n}$ by
$F':\R^{n}\to\R^{n}$, and observe that $F'$ satisfies the conditions
of Theorem \ref{thm:newFTAG.Parallelism-n+1dir}, and so $F'$ is
affine-additive. The theorem's assumptions on $\bar{p}_{n+2}$ imply
that $F'$ also satisfies the conditions of Lemma \ref{lem:Add1Str},
and thus it is affine. By Fact \ref{fact:PrScale}, we may assume
without loss of generality that $F'$ is the identity mapping. Next,
we explain why $F|_{\R P^{n-1}}$ is also the identity. Let $\bar{p}\in\R P^{n-1}$
and take a projective line $\bar{l}$ including $\bar{p}$ and $\bar{p}_{n+2}$.
Since the restriction of $F$ to the affine copy of $\R^{n}$ is the
identity mapping, any point different than $\bar{p}$ is mapped to
itself. Since, by assumption, $\bar{l}$ is mapped onto a line, it
follows that $\bar{p}$ must be mapped to itself. Thus, up to composition
with projective-linear maps, $F$ is the identity map, and thus it
is a projective-linear map.
\end{proof}

\begin{rem}
From Theorem \ref{thm:Pr+1} and Theorem \ref{thm:Pr+2}, one may
deduce a similar result for the unit sphere $\Sph^{n}\sub\R^{n+1}$;
Let $f:\Sph^{n}\to\Sph^{n}$ be an injective mapping which maps any
great circle containing a point of a given set of $n+2$ points (for
example, in general position or with $n+1$ in general position in
$\Sph^{n-1}\subset\Sph^{n}$ and another point not in $\Sph^{n-1}$,
according to Theorems \ref{thm:Pr+1} and \ref{thm:Pr+2}) onto a
great circle. Then $f$ is induced by a linear map $A\in\GL_{n+1}\left(\R\right)$.
Indeed, through any point $x\in\Sph^{n}$ pass at least two great
circles which are mapped onto great circles, and hence $f(x)=-f(-x)$.
Then, we may glue $x$ to $-x$ and induce an injective mapping on
$\R P^{n}$ which satisfies the conditions of either Theorem \ref{thm:Pr+1}
or Theorem \ref{thm:Pr+2}. 
\end{rem}

\section{What happens under a continuity assumption\label{sec:continuousFTAGproof}}

In this section we discuss our previously obtained results, when we
add to the assumptions of the theorems, a continuity assumption. Clearly,
wherever affine-additivity was deduced, a continuity assumption implies
affine-linearity (in fact, even weaker restrictions such as measurability
or local boundness would imply affine-linearity). 

What is less clear, and which we prove below in Proposition \ref{prop:Continuity_affine}
, is that under a continuity assumption, the condition that a mapping
$F:\R^{n}\to\R^{n}$ maps certain lines \textsl{onto} lines may be
replaced by a general collinearity assumption, namely that lines are
mapped \textsl{into} lines. However, to prove this fact we also need
to assume that $ $$F$ is bijective (and not only injective as was
assumed so far). This fact will follow from Brouwer's famous invariance
of domain theorem, which states that any injective continuous mapping
from $\R^{n}$ to $\R^{m}$ is an open mapping (see e.g., \cite[Corollary 19.8]{Bred97}).
The same observation also holds for the projective setting, which
we address in Proposition \ref{prop:Continuity_proj} below.

The above observation can be applied to all the main results in this
note. For example, the continuous versions of Theorems \ref{thm:Poly_n+1},
\ref{thm:NFTAG3D}, and \ref{thm:Pr+2} can be respectively formulated
as follows. We leave it to the reader to combine all other results
with Propositions \ref{prop:Continuity_affine} and \ref{prop:Continuity_proj}
to obtain their continuous versions.
\begin{thm}
Let $m,n\geq2$. Let $v_{1},\dots v_{n+1}\in\R^{n}$ be in general
position. Let $F:\R^{n}\to\R^{n}$ be a continuous bijective mapping
that maps each line in $\mathcal{L}\left(v_{1},\dots,v_{n+1}\right)$
into a line. Then $F$ is a polynomial mapping of the form 
\begin{align*}
\left(F\circ A\right)\left(x_{1},\dots,x_{n}\right) & =\sum_{\delta\in\left\{ 0,1\right\} ^{n}}u_{\delta}\prod_{i=1}^{n}x_{i}^{\delta_{i}}
\end{align*}
where $A\in GL_{n}\left(\R\right)$ and $u_{\delta}\in\R^{n}$ satisfy
the following conditions: 
\begin{itemize}
\item $u_{\delta}=0$ for all $\delta$ with $|\delta|\ge\frac{n+2}{2}$,
and 
\item for each $2\le k<\frac{n+2}{2}$ and every $0\le l\le k-2$ indices
$1\le i_{1}<\cdots<i_{l}\le n$, 
\[
\sum_{\substack{|\delta|=k,\\
\delta_{i_{1}}=\cdots=\delta_{i_{l}}=1
}
}u_{\delta}=0.
\]

\end{itemize}
\end{thm}

\begin{thm}
Let $v_{1},v_{2},\dots,v_{5}\in\R^{3}$ be $3$-independent. Let $F:\R^{3}\to\R^{3}$
be a continuous bijective mapping that maps each line in $\mathcal{L}\left(v_{1},\dots,v_{5}\right)$
into a line. Then $F$ is affine. 
\end{thm}

\begin{thm}
Let $n\geq2$. Let $\bar{p}_{1},...,\bar{p}_{n},\bar{p}_{n+1}\in\R P^{n}$
be generic and suppose that a point $\bar{p}_{n+2}\in\R P^{n}$ satisfies
that $\bar{p}_{n+2}\not\in\Psp\{\bar{p}_{1},...,\bar{p}_{n}\}$ and
$\bar{p}_{n+2}\neq\bar{p}_{n+1}$. Let $F:\R P^{n}\to\R P^{n}$ be
a continuous bijective mapping that maps any projective line containing
one of the points $\bar{p}_{1},...,\bar{p}_{n+2}$ into a projective
line. Then $F$ is a projective-linear mapping.
\end{thm}
We now state and prove the ingredient which enables us to prove the
above theorems based on the ones which we already proved:
\begin{prop}
\label{prop:Continuity_affine}Let $F:\R^{n}\to\R^{n}$ be a continuous
bijective mapping. Suppose that $F$ maps a certain line $l$ into
a line. Then $F$ maps $l$ onto a line.\end{prop}
\begin{proof}
Since $F$ carries $l$ into a line, we may view its restriction to
$l$ as a real valued function defined on the real line. Since it
is continuous and injective, it is also an open mapping, and therefore
$F(l)$ is an open interval. Assume that $F(l)$ is not a (full) line.
Then, there exists an endpoint $y\in\R^{n}$ of the open interval
$F(l)$ which is not attained as the image of a point in $l$. Since
$F$ is onto, there exists a point $x\not\in l$ such that $F(x)=y$.
Let $A$ be an open set around $x$ satisfying that $A\cap l=\emptyset$.
By the invariance of the domain theorem, $F$ is an open mapping,
and thus $F\left(A\right)$ is an open set satisfying that $F\left(A\right)\cap F\left(l\right)=\emptyset$,
a contradiction to the fact that $y\in F\left(A\right)$. Therefore
$F$ maps $l$ onto a line.
\end{proof}
The following proposition deals with the projective continuous case:
\begin{prop}
\label{prop:Continuity_proj}Let $F:\RP^{n}\to\RP^{n}$ be a continuous
bijective mapping. Suppose that $F$ maps a certain projective line
$l$ into a projective line. Then $F$ maps $l$ onto a projective
line.\end{prop}
\begin{proof}
As Brouwer's invariance of domain theorem is a local statement, it
holds for general manifolds (without boundary), in particular for
$\RP^{n}$. Thus, the proof of Proposition \ref{prop:Continuity_proj}
is literally the same as the proof of Proposition \ref{prop:Continuity_affine}.
\end{proof}

\section{The fundamental Theorem - An historical account\label{sec:history}}

We have not found an accessible account of the various forms and generalizations
of the fundamental theorems of affine and projective geometry. In
this section we try to produce a list of known results for comparison
with our results and for future references.

For the simplicity of the exposition we shall state the results for
the field $\R$, and indicate, together with references, when they
are also valid for other fields such as $\C$ and $\Z_{p}$. In Section
\ref{sec:OtherFields} below we shall briefly discuss more general
underlying structures for which such results hold.

\subsection{\label{sec:History}Classical theorems of affine and projective geometry}

It seems that the earliest appearance of the fundamental theorems
in the literature was for the real projective plane and goes back
to Von Staudt (1847), see e.g., \cite[page 38]{Coxeter}. Perhaps
the most familiar modern version is the following, which is implied
by Theorem \ref{thm:Artin} below (\cite[Theorem 2.26]{Art57}) by
setting both underlying division rings to be $\R$.
\begin{thm}
Let $n\ge3$. Let $F:\RP^{n}\to\RP^{n}$ be a bijection. Assume that
$F$ takes any three collinear points into collinear points. Then
$F$ is projective linear.
\end{thm}
The most classical version of the fundamental theorem of affine geometry,
is a simple consequence of its projective counterpart, and states
the following (see e.g., \cite[page 52]{Berger}, letting the underlying
fields both be $\R$)
\begin{thm}
\label{thm:ClassFTAG}Let $n\ge2$. Let $F:\R^{n}\to\R^{n}$ be a
bijection. Assume that $F$ takes any three collinear points into
collinear points. Then $F$ is affine. 
\end{thm}
Theorem \ref{thm:ClassFTAG} holds for other fields, such as $\C,\Z_{p}\left(p\neq2\right)$,
and even division rings. However, in such general cases, semi-affine
maps should be considered instead of affine maps. For $\Z_{2}$, the
theorem does not hold in general, as was observed in \cite[Remark 12]{Chub99}.
In the same paper, the authors completely analyze this case.

\subsection{Without surjectivity }

One may remove the surjectivity assumption from the fundamental theorems,
and replace them with other mild conditions. The first example is
obtained by replacing surjectivity by the condition that lines are
mapped \textit{onto} lines (as is assumed in the main results in this
note), see e.g., \cite[page 925]{Lester95}:
\begin{thm}
Let $n\ge2$. Let $F:\R^{n}\to\R^{n}$ be injective. Assume that $F$
takes any line onto a line. Then $F$ is affine. 
\end{thm}
Another useful result, and much more general, is stated in \cite[page 122]{Grub92Bodies},
as an easy consequence of Hilfssatz $3$ of Lenz \cite{Lenz} :
\begin{thm}
Let $n\ge2$. Let $F:\R^{n}\to\R^{n}$ be injective. Assume that $F$
takes any three collinear points into collinear points. Also assume
that $F\left(\R^{n}\right)$ is not contained in a line. Then $F$
is affine. 
\end{thm}

\subsection{Without injectivity}

In a work of Chubarev and Pinelis \cite{Chub99}, the authors proved
that the injectivity assumption can be removed from the fundamental
theorem. They also generalize the collinearity condition to having
$q$-planes mapped into $q$-planes for some $q\in\left\{ 1,\dots n-1\right\} $,
where a $q$-plane is a translate of a $q$-dimensional subspace.
Moreover, their results also hold in a general setting of vector spaces
over division rings. They also carefully treat the case where on of
the underlying fields is $\Z_{2}$. For $\R$ their main result reads
the following.
\begin{thm}
Let $n\ge m\ge2$. Let $F:\R^{n}\to\R^{m}$ be surjective. Let $q\in\left\{ 1,\dots,n-1\right\} $.
Assume that $F$ takes every $q$-plane in $\R^{n}$ into a $q$-plane.
Then $F$ is affine.
\end{thm}

\subsection{Collinearity in a limited set of directions }

It seems that results in the spirit of this paper, where one restricts
the family of lines for which collinearity is preserved, have been
considered in the literature mainly for dimension $n=2$. For the
case of the real projective plane, Kasner \cite{Kas1903} proved that
a twice differentiable self-map is projective-linear if it maps each
line in a ``$4$-web'' family of lines into a line, where a ``$4$-web''
consists of four pairwise transversal families of lines, each covering
the domain of the map. Later on, in the 1920\textquoteright s, W.
Blaschke and his co-workers stated that this principle is true without
the differentiability assumption (see \cite[p. 91]{Bla44}) and a
complete proof was of this fact was given in 1935 by Prenowitz \cite[Theorem V]{Pre1903}. 

For the case of higher dimension, although the theorem sounds classical,
we have not found this stated anywhere in the literature. One related
result is that of Shiffman \cite[Theorem 3]{Shif95} where he assumes
that collineations are preserved for points on an open (thus huge)
set of lines. However, his result (as well as the aforementioned results
for the projective plane) is in a richer framework concerning only
segments in an open subset. To formally state his result, we need
some notation. Let $ $$\CP^{n}$ denote the complex projective space.
Denote the complex conjugate of a function $f:\CP^{n}\to\CP^{n}$
by $\bar{f}$. Let ${\cal L}_{\R}^{n}$, ${\cal L}_{\C}^{n}$ denote
the set of lines in the projective real and complex $n$-spaces. We
give the projective spaces $\RP^{n}$, $\CP^{n}$ and the Grassmannians
${\cal L}_{\R}^{n}$, ${\cal L}_{\C}^{n}$  the usual metric topologies.
For a subset $U\subset\RP^{n}$ we write $\L\left(U\right)=\left\{ L\in\L_{\R}^{n}\::\:L\cap U\neq\emptyset\right\} $
(and similarly for a subset $U\subset\CP^{n}$). Shiffman proved the
following:
\begin{thm}
\label{thm:Shiffman} Let $n\ge2$. Let $U$ be a connected open set
in $\RP^{n}$ ($\CP^{n}$) and let $\L_{0}$ be an open subset of
$\L\left(U\right)$ such that $U\sub\bigcup\L_{0}$. Suppose that
$f:U\to\RP^{n}$ ($f:U\to\CP^{n}$) is a continuous injective map
such that $f\left(L\cap U\right)$ is contained in a projective line
for all $L\in\L_{0}$. Then there exists a projective-linear transformation
$A$ such that $f=\left.A\right|_{U}$ (and in the complex case: $f=\left.A\right|_{U}$
or $\bar{f}=\left.A\right|_{U}$).
\end{thm}

\subsection{The fundamental theorems on windows }

The classical fundamental theorem of affine geometry, e.g. Theorem
\ref{thm:ClassFTAG}, characterizes self-maps of $\R^{n}$ which map
lines, or segments, into segments. Such maps turn out to be affine.
In the projective case, such maps turn out to be projective-linear.
In Shiffman's Theorem \ref{thm:Shiffman}, the same conclusion holds
when the domain of the transformation is any connected open subset
of the projective $n$-space. His result can be translated to the
affine setting, where projective-linear maps induce a special class
of segment preserving maps when restricted to a subset of $\R^{n}$,
called fractional linear maps. Such maps are defined as follows. Let
$D\subset\R^{n}$ be a domain contained in a half-space. Fix a scalar
product $\iprod{\cdot}{\cdot}$ on $\R^{n}$ and let $A$ be a linear
map, $b,c\in\R^{n}$ two vectors and $d\in\R$ some constant. The
(fractional linear) map
\[
v\to\frac{1}{\iprod cv+d}\left(Av+b\right)
\]
 is defined on the open half-space $\iprod cv<-d$ and is segment
preserving (and injective).

In \cite[Theorem 2.17]{ArFlMi}, the authors prove the following theorem
for convex domains of $\R^{n}$ (or ``windows''). As well as formulating
their result in the affine setting, they use a different approach
than that of Shiffman. 
\begin{thm}
Let $n\ge2$ and let $K\subset\R^{n}$ be a convex set with non-empty
interior. Suppose $F:K\to\R^{n}$ is an injective map which maps each
segment in $K$ into a segment. Then $F$ is a fractional linear map. 
\end{thm}
Many other properties of fractional linear maps are investigated in
\cite{ArFlMi}, as well as sufficient conditions which force fractional
linear maps to be affine.

\subsection{\label{sec:OtherFields}General underlying structures }

Let $V,V'$ be two left vector spaces over fields $k$ and $k'$ respectively.
Assume that there exists an isomorphism $\mu$ which maps $k$ onto
$k'$. Then a map $\lambda:V\to V'$ is called semi-linear with respect
to $\mu$ if $\lambda\left(ax+by\right)=\mu\left(a\right)\lambda\left(x\right)+\mu\left(b\right)\lambda\left(y\right)$
for every $x,y\in V$ and all $a,b\in k$. 

A bijective map $\sigma$ between two projective spaces $\bar{V}$,
$\bar{V'}$ of equal dimension is said to be a collineation if for
any projective subspaces $\bar{U_{1}},\bar{U_{2}}\in\bar{V}$, $\bar{U_{1}}\sub\bar{U_{2}}$
implies $\sigma\bar{U_{1}}\sub\sigma\bar{U_{2}}$. 

Perhaps the most well-known modern variation of the fundamental theorem
of projective geometry, with the most general underlying structure
appears in Artin's book \cite[Theorem 2.26]{Art57}:
\begin{thm}
\label{thm:Artin}Let$ $ $V$ and $V'$ be (left) vector spaces of
equal dimension $n\ge3$ over division rings $k$ respectively $k'$,
and let $\bar{V}$, $\bar{V'}$ be the corresponding projective spaces.
Let $\sigma:\bar{V}\to\bar{V'}$ be a bijective correspondence which
maps collinear points to collinear points. Then there exists an isomorphism
$\mu$ of $k$ onto $k'$ and a semi-linear map $\lambda$ of $V$
onto $V'$ (with respect to $\mu$) such that the collineation which
$\lambda$ induces on $\bar{V}$ agrees with $\sigma$ on the points
of $\bar{V}$. If $\lambda_{1}$ is another semi-linear map with respect
to an isomorphism $\mu_{1}$ of $k$ onto $k'$ which also induces
this collineation then $\lambda_{1}\left(x\right)=\lambda\left(\alpha x\right)$
for some fixed $\alpha\neq0$ of $k$ and the isomorphism $\mu_{1}$
is given by $\mu_{1}\left(x\right)=\mu\left(\alpha x\alpha^{-1}\right)$.
For any $\alpha\neq0$ the map $\lambda\left(\alpha x\right)$ will
be semi-linear and induce the same collineation as $\lambda$. The
isomorphism $\mu$ is, therefore, determined by $\sigma$ up to inner
automorphisms of $k$.
\end{thm}
There exist other variations of the fundamental theorems, mainly concerning
the underlying structure, which we will not state in this note. One
such example can be found in \cite{Mcdonald} for free modules over
local rings.

\begin{appendices}

\section{A fundamental theorem under a Parallelism condition\label{sec:paral}}

In this appendix we prove theorems \ref{thm:newFTAG.Parallelism-n+1dir}
and \ref{thm:newFTAG.Parallelism}. As mentioned before, these results
were stated and proved in a more general setting in \cite{ArtsteinSlomka2011},
yet their proofs in our setting are much simpler. Moreover, one should
also note the simplicity of these results, in comparison to our results
from previous sections.

We begin with the following linear-algebra lemma.
\begin{lem}
\label{lem:parallelism} Let $2\leq n$. Let $v_{1},\ldots,v_{n}$
be linearly independent vectors in $\R^{n}$. Let $F:\R^{n}\rightarrow\R^{m}$
be an injection, $F(0)=0$, and assume $F$ maps each line in $\L(v_{1},\ldots,v_{n})$
onto a line, and moreover, $F(\L(v_{i}))=\L(F(v_{i}))$. That is,
parallel lines in the family are mapped onto parallel lines. Then
the following holds for every $2\le k\le n$. 
\begin{eqnarray}
 & F(v_{1}),F(v_{2}),\dots,F(v_{k})\text{ are linearly independent},\label{eq:1*}\\
 & F(\sp\{v_{1},\dots,v_{k}\})=\sp\{F(v_{1}),\dots,F(v_{k})\}.\label{eq:2*}
\end{eqnarray}
Moreover, 
\begin{align}
F(v_{k}+\sp\{v_{1},\dots,v_{k-1}\})=F(v_{k})+\sp\{F(v_{1}),\dots,F(v_{k-1})\}\label{eq:3*}
\end{align}
\end{lem}
\begin{proof}
[Proof of Lemma \ref{lem:parallelism}]By assumption, 
\begin{equation}
F(x+\sp\{v_{i}\})=F(x)+\sp\{F(v_{i})\}\label{eq:web-to-web}
\end{equation}
for all $v_{i}$ and $x\in\R^{n}$.

Next, we proceed by induction on $k$ to prove that (\ref{eq:1*})
and (\ref{eq:2*}) hold. For $m=1$ the claim is trivial. Assume that
(\ref{eq:1*}) and (\ref{eq:2*}) hold for $(k-1)$. Assume that 
\[
F(v_{k})\in\sp\{F(v_{1}),...,F(v_{k-1})\}.
\]
Then, the fact that (\ref{eq:2*}) holds for $k-1$ implies that there
exists $u\in\sp\{v_{1},...,v_{k-1}\}$ such that $F(u)=F(v_{k})$.
The injectivity of $F$ implies that $u=v_{k}$, which contradicts
the fact that $v_{1},...,v_{k}$ are linearly independent.

Next we show that (\ref{eq:2*}) holds for $k$. Denote the projection
onto $\sp\{v_{2},...,v_{k}\}$ along $v_{1}$ by $P_{1}$. Let $x\in\sp\{v_{1},...,v_{k}\}$.
Since $x-P_{1}x\in\sp\{v_{1}\}$ it follows from (\ref{eq:web-to-web})
that 
\[
F(x)\in F(P_{1}x)+\sp\{F(v_{1})\}.
\]
By the induction hypothesis we have 
\[
F(P_{1}x)\in\sp\{F(v_{2}),...,F(v_{k})\},
\]
and so $F(x)\in\sp\{T(v_{1}),...,T(v_{k})\}.$ Thus, $F(\sp\{v_{1},\dots,v_{k}\})\subset\sp\{F(v_{1}),\dots,F(v_{k})\}.$
For the opposite direction, pick a point $y\in\sp\{F(v_{1}),\dots,F(v_{k})\}$.
Take the line through $y$ which is parallel to $F(v_{1})$. This
line must intersect the subspace $\sp\{F(v_{2}),\dots,F(v_{k})\}$
at some point, which by the induction hypothesis is $F(z)$ with $z\in\sp\{v_{2},\dots,v_{k}\}$.
Take the line parallel to $v_{1}$ and passing through $z$. By assumption,
this line is mapped onto the aforementioned line, and thus $y$ is
in $F(\sp\{v_{1},\dots,v_{k}\})$ as required.

Here one should readily notice that the image of $F$ is in fact a
subspace of dimension $n$, and so without loss of generality $m=n$
and $F$ is a bijection. Moreover, $F^{-1}$ maps all lines in $\L(F(v_{1}),\ldots F(v_{n}))$
onto lines, and parallel lines of this family onto parallel lines.

It remains to prove that property (\ref{eq:3*}) holds. For each $i=1,...,k$
denote the projections onto $\sp\{e_{j}\}_{j\neq i}$ along $v_{i}$
by $P_{i}$. Let $x\in v_{k}+\sp\{v_{1},\dots,v_{k-1}\}$ and recursively
define $y_{0}=x$ and $y_{i}=P_{i}y_{i-1}$. Obviously, we have $y_{i-1}-y_{i}\in\sp\{v_{i}\}$
for $1\leq i\leq m-1$ and $y_{k-1}=v_{k}$. Thus, $F(y_{i-1})-F(y_{i})\in\sp\{F(v_{i})\}$
for each $i=1,\dots,k-1$. By writing, 
\[
F(x)=\left[F(x)-F(y_{0})\right]+\left[F(y_{0})-F(y_{1})\right]+\dots+\left[F(y_{k-2})-F(y_{k-1})\right]+T(y_{k-1}),
\]
we obtain $F(x)\in F(v_{k})+\sp\{F(v_{1}),\dots,F(v_{k-1})\}$ and
hence 
\[
F(v_{k}+\sp\{v_{1},\dots,v_{k-1}\})\subset F(v_{k})+\sp\{F(v_{1}),\dots,F(v_{k-1})\}.
\]
Applying the same reasoning for $F^{-1}$ proves the equality.
\end{proof}

Next we prove Theorem \ref{thm:newFTAG.Parallelism}:
\begin{proof}
[Proof of Theorem \ref{thm:newFTAG.Parallelism}]Without loss of generality
(by Fact \ref{fact:trans} and Fact \ref{fact:linear}) we may assume
that $F(0)=0$, $\{v_{i}\}_{i=1}^{n}=\{e_{i}\}_{i=1}^{n}$ is the
standard basis of $\R^{n}$, and that $F(e_{i})=e_{i}$. Here we use
the fact that $F(v_{i})$ are linearly independent, which follows
from Lemma \ref{lem:parallelism}. These assumptions will result in
the extra linear factors $A$ and $B$ in the statement of the theorem.

For each $i=1,\dots,n$ and every $a\in\R$ define $f_{i}:\R\rightarrow\R$
by 
\[
F(ae_{i})=f_{i}(a)e_{i}.
\]
Let $x=(x_{1},x_{2},...,x_{n})^{T}\in\R^{n}$. By Lemma \ref{lem:parallelism}
we have 
\[
F(x)\in F(x_{i}e_{i}+\sp_{j\neq i}\{e_{j}\})=F(x_{i}e_{i})+\sp_{j\neq i}\{F(e_{j})\}=f_{i}(x_{i})e_{i}+\sp_{j\neq i}\{e_{j}\}.
\]
Hence, 
\[
(F(x))_{i}=f_{i}(x_{i})
\]
and so 
\[
F(x)=\sum_{i=1}^{n}f_{i}(x_{i})e_{i},
\]
as required. The fact that the $f_{i}$'s are bijective trivially
holds since $F$ maps lines in $\L\left(e_{1},\dots,e_{n}\right)$
onto lines. 
\end{proof}
Finally, we prove Theorem \ref{thm:newFTAG.Parallelism-n+1dir}:
\begin{proof}
[Proof of Theorem \ref{thm:newFTAG.Parallelism-n+1dir}] Without loss
of generality (by Fact \ref{fact:trans} and Fact \ref{fact:linear})
we may assume that $F(0)=0$, $\{v_{i}\}_{i=1}^{n}=\{e_{i}\}_{i=1}^{n}$
is the standard basis of $\R^{n}$, $v_{n+1}=v=e_{1}+\cdots+e_{n}$.
Here we use the fact that $v_{1},\dots,v_{n+1}$ are $n$-independent.
These assumptions will result in the extra linear factor $A$ in the
statement of the theorem. Similarly, we may assume without loss of
generality that  $F(e_{i})\in\sp\left\{ e_{i}\right\} $ for all $i=1,\dots,n$
and that $F\left(v\right)=v$. Here we use the fact that $F\left(v_{1}\right),\dots,F\left(v_{n+1}\right)$
are $n$-independent, which follows from Lemma \ref{lem:parallelism}.
These assumptions will result in the extra linear factors $B$ in
the statement of the theorem. 

By Theorem \ref{thm:newFTAG.Parallelism}, there exist bijective functions
$f_{1},\dots f_{n}:\R\to\R$ such that
\[
F\left(x\right)=\left(f_{1}\left(x_{1}\right),\dots,f_{n}\left(x_{n}\right)\right)^{T}
\]
for every $x=\left(x_{1},\dots,x_{n}\right)^{T}\in\R^{n}$. 

Let $t\in\R$. since $F$ carries the line $\sp\left\{ v\right\} $
onto itself, it follows that $F\left(tv\right)=\left(f_{1}\left(t\right),\dots,f_{n}\left(t\right)\right)^{T}\in\sp\left\{ v\right\} $
and thus $f_{1}\left(t\right)=\cdots=f_{n}\left(t\right)$ for all
$t\in\R$. In other words, $f_{1},\dots,f_{n}$ are identical, and
are denoted from here on by $f$. 

Next we show that $f$ is an additive function. Let $a,b\in\R$. Since
the line passing through $be_{1}$and $av+be_{1}$ is parallel to
$v$, and since $F$ maps all lines in $\L\left(v\right)$ onto lines
parallel to $F\left(\sp\left\{ v\right\} \right)=\sp\left\{ v\right\} $,
it follows that 
\[
F\left(av+be_{1}\right)-F\left(be_{1}\right)=F\left(\left(a+b\right)e_{1}+ae_{2}+\cdots ae_{n}\right)-F\left(be_{1}\right)\in\sp\left\{ v\right\} .
\]
By the above representation of $F$, together with the fact that $F\left(0\right)=0$,
it follows that 
\[
\left(f\left(a+b\right),f\left(a\right),f\left(a\right),\dots,f\left(a\right)\right)-\left(f\left(b\right),0,0,\dots,0\right)\in\sp\left\{ v\right\} ,
\]
and thus $f\left(a+b\right)-f\left(b\right)=f\left(a\right)$, as
claimed.
\end{proof}
\end{appendices}

\bibliographystyle{amsplain}
\bibliography{BibNewFTAG_April_6_2016}

\end{document}